\documentclass[10pt,twocolumn,letterpaper]{article}

\usepackage[utf8]{inputenc} 
\usepackage[T1]{fontenc}
\usepackage{url}
\usepackage{ifthen}
\usepackage{cite}
\usepackage[cmex10]{amsmath} 
\usepackage[top=0.7in,bottom=0.7in,left=0.6in,right=0.7in,columnsep=0.35in]{geometry}
\usepackage{authblk}

\usepackage{xcolor}
\usepackage{amsthm}
\usepackage{amssymb}
\usepackage{mathtools}
\usepackage{float}

\newtheorem{lemma}{Lemma}
\newtheorem{theorem}{Theorem}

\begin{document}
\title{Information-Theoretic Bounds for Integral Estimation} 

\author[1]{Donald Q.~Adams}
\author[2]{Adarsh~Barik}
\author[3]{Jean~Honorio}

\affil[1]{Department of Computer Science\\
                    Purdue University\\
                    West Lafayette, IN 47907, USA\protect\\
                    Email: adams391@purdue.edu}
\affil[2]{Department of Computer Science\\
                    Purdue University\\
                    West Lafayette, IN 47907, USA\protect\\
                    Email: abarik@purdue.edu}
\affil[3]{Department of Computer Science\\
                    Purdue University\\
                    West Lafayette, IN 47907, USA\protect\\
                    Email: jhonorio@purdue.edu}

\date{}
\maketitle

\begin{abstract}
   In this paper, we consider a zero-order stochastic oracle model of estimating definite integrals. In this model, integral estimation methods may query an oracle function for a fixed number of noisy values of the integrand function and use these values to produce an estimate of the integral. We first show that the information-theoretic error lower bound for estimating the integral of a $d$-dimensional function over a region with $l_\infty$ radius $r$ using at most $T$ queries to the oracle function is $\Omega\left(2^d r^{d+1}\sqrt{d/T}\right)$. Additionally, we find that the Gaussian Quadrature method under the same model achieves a rate of $O\left(2^{d}r^d/\sqrt{T}\right)$ for functions with zero fourth and higher-order derivatives with respect to individual dimensions, and for Gaussian oracles, this rate is tight. For functions with nonzero fourth derivatives, the Gaussian Quadrature method achieves an upper bound which is not tight with the information-theoretic lower bound. Therefore, it is not minimax optimal, so there is space for the development of better integral estimation methods for such functions.
\end{abstract}

\section{Introduction}
\label{sec:introduction} 

Estimating definite integrals is a common technique used in many fields. Methods such as the Trapezoid Rule and Simpson’s Rule are taught in introductory calculus courses as fundamental approaches to approximating the value of definite integrals by simply querying the function at certain specified points. Now, methods such as the Gaussian Quadrature method are used in physics to perform integral estimation in the Finite Element Method, and in the fields of statistics and machine learning, integral estimation often arises when computing expectations of intractable functions.

Due to the commonality of this technique, one subsequent inquiry would be to determine the hardness of correctly estimating such definite integrals with respect to factors such as dimension, the region of integration, and the number of times the function can be queried as part of the estimation method. Using a minimax approach, lower bounds can be obtained on the minimum possible error attainable by any integral estimation method. Likewise, statistical complexity upper bounds can be determined for the error of specific integral estimation methods, which allows them to be compared to the information-theoretic lower bounds to find areas of improvement. By assuming that queries to the function return a noisy value provided by an oracle, computing an accurate integral estimation both becomes harder and becomes prone to statistical analysis.

Thus far, it seems that no other papers have analyzed this problem from an information-theoretic point of view. Integral estimation is often performed in the context of computing expectations, so the results are relevant to specific probability density functions. In this paper, we instead consider a more general class of functions. Additionally, some particular integral estimation methods have been analyzed to determine bounds on the error given precise, non-noisy values of the integrand function \cite{devore1984error}. However, the added obstacle of a noisy oracle function makes the upper bounds from this paper more general, and the information-theoretic minimax approach allows the determined lower bounds to hold generally over arbitrary integral estimation methods. 

In this paper, we achieved the following results. First, we show that the information-theoretic error lower bound for any integral estimation method on $d$-dimensional functions over a region containing an $l_\infty$-norm ball of radius $r$ using $T$ queries to the oracle function is $\Omega\left(2^d r^{d+1}\sqrt{d/T}\right)$. We then prove that the Gaussian Quadrature method of integral estimation converges with a rate of $O\left(2^{d}r^d/\sqrt{T}\right)$ for functions with zero fourth and higher-order derivatives with respect to individual dimensions where each point determined by the Gaussian Quadrature method is queried $m$ times. By noting that $T=m2^d$ in this case, we see that the upper bound of Gaussian Quadrature is almost tight with the general lower bound. Finally, the Simpson’s Rule method achieves a similar upper bound with respect to $m$, but it queries a greater number of overall points to achieve it.

\section{Preliminaries}
\label{sec:preliminaries} 

In this section, we lay out our problem and provide definitions of the different error types and integration methods which will be discussed. Assume that we are given a region of integration $S\subseteq\mathbb{R}^d$ and a set of $d$-dimensional functions $\mathcal{F}=\{f\mid f:S\rightarrow\mathbb{R}\}$. We wish to estimate the integral $\int_{x\in S} f(x)\in\mathbb{R}$ using only $T$ queries to a zero-order stochastic oracle. A zero-order stochastic oracle is defined as a random function $\phi:S\rightarrow\mathbb{R}$ that returns a noisy, unbiased estimate $\phi(x, f)$ of the function $f$ with bounded variance. That is,

$$\textbf{E}[\phi(x, f)] = f(x)\text{ and Var}(\phi(x, f)) \leq \sigma^2$$

Additionally, a model $\mathcal{M}$ is defined as a method which makes $T$ queries to the given oracle function at points $x_1,\ldots,x_T$ defined by the model and returns an estimate for the integral $\int_{x\in S} f(x)$ using the noisy values returned by the oracle function. Then we let $\mathbb{M}_T$ be the class of all models that meet the above definition.

\subsection{Expected Risk and Minimax Error}

In this section, we discuss the definitions for the expected risk of an individual model and the minimax error used to determine the theoretical best achievable error across all possible models as often used in statistics and machine learning \cite{wainwright00,wasserman2006all}. For any particular function $f$ and model $\mathcal{M}$, let 

\begin{equation}
    I=\int_{x\in S} f(x)
\end{equation}

be the true integral value of the function $f$ over the integration region $S$, and let $I_\mathcal{M}$ be the estimated integral value produced by the model. Then the expected risk of the model $\mathcal{M}$, which is the average error of the model over all oracle functions, is defined as $\textbf{E}_\phi\left[\left|I-I_\mathcal{M}\right|\right]$. Now we define the maximum risk of the model by analyzing the expected risk over all functions $f\in\mathcal{F}$. That is,

$$\epsilon(\mathcal{M},\mathcal{F},\phi)\coloneqq\sup_{f\in\mathcal{F}}\textbf{E}_\phi\left[\left|I-I_\mathcal{M}\right|\right]$$

Finally, in order to determine the minimax error, we consider the model with the lowest maximum risk. That is, the minimax error is defined as the infimum of the maximum risk over all models in the class $\mathbb{M}_T$. 

$$\epsilon^*(\mathcal{F},\phi)\coloneqq\inf_{\mathcal{M}\in\mathbb{M}_T}\epsilon(\mathcal{M},\mathcal{F},\phi)$$

Therefore, we can use this minimax approach to find a lower bound on the theoretical best attainable error by any possible model. The main result of this paper will determine such a lower bound by considering a subclass of functions, so since the maximum risk of a model takes the supremum over all functions $f\in\mathcal{F}$, a lower bound for any subclass of functions $\mathcal{G}\subset\mathcal{F}$ must also be a lower bound for the minimax error of $\mathcal{F}$.

\subsection{Gaussian Quadrature Method}

We will also analyze the sample complexity upper bound for the Gaussian Quadrature integral estimation method. With this method, we integrate over a region $R=[-r,r]^d$, and we let $V_d=(\pm\frac{r}{\sqrt{3}}, \pm\frac{r}{\sqrt{3}},\ldots,\pm\frac{r}{\sqrt{3}})$ be the set of points at which this method will query the oracle function. Clearly, $\left|V_d\right|=2^d$.

Then let $T=m2^d$ be the total number of times the Gaussian Quadrature method will query the oracle function. More specifically, the method will query the oracle function $m$ times for each point $v\in V_d$ and will take the average of these values. Since the oracle function gives a noisy, but unbiased, estimate of the function, querying an individual point multiple times will reduce the variance that the oracle function imposes on the integral estimation. 

Now, when computing the estimation of the integral, let $\hat{f}(x)=\frac{1}{m}\sum_{i=1}^m\phi(x,f)$. Then the Gaussian Quadrature method uses the following formula to estimate the value of the integral.

$$\int_{x\in R} f(x) \approx r^d\sum_{v\in V_d} \hat{f}(v)$$

Note that, in the one-dimensional case, the Gaussian Quadrature method is exact for polynomials of degree up to 3. In the multi-dimensional case, we find that this formula is exact for polynomials which do not exceed degree up to 3 for any individual dimension. For instance, integrating the function $f(x_1,x_2,\ldots,x_d)=x_1^3 x_2^3\ldots x_d^3$ would still be exactly estimated by the multi-dimensional Gaussian Quadrature method.

\section{Main Results}
\label{sec:mainresults} 

In this section, we lay out the main results. Theorem \ref{theorem:general_lower_bound} is the main result of the paper and demonstrates the information-theoretic lower bound of the error of estimating integrals. Since the minimax error definition takes the supremum over all functions in the function class, the proof of this lower bound relies on the construction of a subclass of functions for which the lower bound holds, thus proving that the lower bound holds for the general class of functions as well. This approach of using a restricted ensemble is customary for information-theoretic lower bounds \cite{santhanam00,wang00,tandon00,ke00}.

The information-theoretic lower bound we achieved for the minimax error is $\Omega\left(2^d r^{d+1}\sqrt{d/T}\right)$ where $d$ is the number of dimensions of the input space, $r$ is the radius of an $l_\infty$ ball contained in the integration region, and $T$ is the number of queries a method may make to the oracle function. This means that, for an integration region of fixed size and dimension, there cannot exist a method of estimating integrals that achieves a convergence rate in $T$ faster than $\Omega(\sqrt{1/T})$.

\begin{theorem}\label{theorem:general_lower_bound}
For any class of $d$-dimensional functions $\mathcal{F}=\{f\mid f:S\rightarrow\mathbb{R}\}$ and any region of integration $S\subseteq\mathbb{R}^d$ containing an $l_\infty$ ball of radius $r$, there exists a constant $c$ such that the minimax error of estimating the integral using at most $T$ queries is upper bounded as 

$$\epsilon^*(\mathcal{F},\phi)\geq c2^dr^{d+1}\sqrt{\frac{d}{T}}$$

provided that $\sup_{x^*\in S}||x^*||_\infty\leq 2\sigma$ where $\sigma^2$ is the upper bound on the variance of the oracle.
\end{theorem}

Next, we find that the sample complexity upper bound for the Gaussian Quadrature method is $O\left(2^dr^d\sigma/\sqrt{T}+2^dr^{d+5}\right)$ for functions with nonzero fourth derivatives with respect individual dimensions and $O\left(2^dr^d\sigma/\sqrt{T}\right)$ when the functions have zero fourth and higher-order derivatives. Thus we can conclude that, for functions with nonzero fourth and higher-order derivatives, the Gaussian Quadrature method does not achieve a tight upper bound, so it is possible to develop better integral estimation methods. (The Simpson's Rule method was also analyzed and achieved similar results).

\begin{theorem}\label{theorem:gq_upper_bound}
  If $\left|f_i^{(4)}(x)\right|\leq K$ for all $i\in\{1,\ldots,d\}$ and $x\in[-r,r]^d$, then the error for the Gaussian Quadrature method has the following upper bound.
  
  $$\epsilon(GQ,F,\phi)\leq \frac{2^{d+1}r^d\sigma}{\sqrt{T}}+\frac{2^{d+1} r^{d+5}}{6\cdot 45}K$$
  
  Likewise, if $K=0$, then we get the following upper bound.
  
  $$\epsilon(GQ,F,\phi)\leq \frac{2^{d+1}r^d\sigma}{\sqrt{T}}$$

\end{theorem}

In order to prove Theorem \ref{theorem:gq_upper_bound}, we first proved two additional lemmas. First, using the fact that the one-dimensional Gaussian Quadrature method is exact for polynomials of degree up to 3, we determine that the multi-dimensional extension of Gaussian Quadrature is also exact for polynomials of degree up to 3 with respect to each dimension. Note that this holds when the method has access to non-noisy values of the integrand function.

\begin{lemma}\label{lemma:gq_exact}
If f is a polynomial of degree at most 3, then a non-noisy estimation from the Gaussian Quadrature method will exactly estimate the integral. That is,
    $$\int_{x\in R} f(x)=r^d \sum_{v\in V_d}f(v)$$
\end{lemma}

Additionally, once we know that Gaussian Quadrature is exact for polynomials of degree up to 3, we then use the error formula for Hermite Interpolation to find an error term for non-noisy Gaussian Quadrature estimations. This formula allows us to find an upper bound on the sample complexity error of the Gaussian Quadrature method when it only has access to noisy function values, thus proving Theorem \ref{theorem:gq_upper_bound}.

\begin{lemma}\label{lemma:gq_error_term}

A non-noisy estimation from the Gaussian Quadrature method will achieve an error term with the following upper bound.

  $$\left|\int_{x\in R} f(x)-r^d\sum_{v\in V_d}f(v)\right|\leq \frac{c2^d r^d}{4!}\sup_{i,x^*}\left|f_i^{(4)}(\xi(x^*))\right|$$
  
In this bound, the supremum is considering the maximum fourth derivative with respect to the $i^{th}$ dimension where $\xi(x)$ is a point determined by the error formula for Hermite Interpolation. 

\end{lemma}

Finally, we consider the Gaussian Quadrature method for a Gaussian oracle function with variance $\sigma^2$, and we determine that, for functions with zero fourth and higher-order derivatives, the above rates for Gaussian Quadrature are tight. As such, for a fixed variance Gaussian oracle, we have that the Gaussian Quadrature method achieves an error rate that is tight with the information-theoretic lower bound. This formula for the Gaussian Quadrature error with a Gaussian oracle was also verified experimentally by generating random polynomials and computing the average error produced by the Gaussian Quadrature method for different values of $T$, the number of queries.

\begin{theorem}\label{theorem:gq_gaussian_oracle}
  Let $\phi$ be a Gaussian oracle function with variance $\sigma^2$. If $\left|f_i^{(4)}(x)\right|\leq K$ for all $i\in\{1,\ldots,d\}$ and $x\in[-r,r]^d$, then the error for the Gaussian Quadrature method with oracle function $\phi$ has the following formula
  
  \begin{align*}\epsilon(GQ,F,\phi) = &\frac{2^{d}r^d\sigma}{\sqrt{T}}\exp\left(-\frac{c_1 2^d }{\sigma^2}\right)\sqrt{\frac{2}{\pi}}+\\
  &\frac{c_2 2^{3d/2} r^d}{\sqrt{T}}\text{erf}\left(\frac{c_3 2^{d/2}}{\sigma}\right)
  \end{align*}
  
  where erf is the Gauss error function and $c_1$, $c_2$, and $c_3$ are constants such that $c_1\in[0,r^{10}/(6\cdot 45)]$, $c_2\in[0,r^5/(3\cdot 45)]$, and $c_3\in[0,r^5/(3\cdot 45\sqrt{2})]$.
  
  Likewise, if $K=0$, then we get the following explicit formula for the error.
  
  $$\epsilon(GQ,F,\phi)=\frac{2^{d}r^d\sigma}{\sqrt{T}}\sqrt{\frac{2}{\pi}}$$

\end{theorem}

\section{Proof of Theorem \ref{theorem:general_lower_bound}: Infor\-mation-Theoretic Lower Bound}
\label{sec:proof}

In this section, we provide the proof for Theorem \ref{theorem:general_lower_bound}. Since the theorem holds for a class of functions $\mathcal{F}$, we start by defining a subclass of functions $\mathcal{G}(\delta,h)$ that are parameterized by a discrete set of vectors. Then by proving that estimating the integral is as hard as determining the discrete-valued parameters, we can apply Fano's Inequality to get a lower bound on the subclass of functions $\mathcal{G}(\delta,h)$, which must, therefore, hold as a lower bound for the general class of functions $\mathcal{F}$.

\subsection{Defining Function Space}

We define $\mathcal{V}\subseteq\{-1,+1\}^d$ such that, for any $\alpha,\beta\in\mathcal{V}$, if $\alpha\ne\beta$, then

$$\sum_{i=1}^d 1\left[\alpha_i\ne\beta_i\right]\geq d/4$$

Then it is possible to construct a set $\mathcal{V}$ with cardinality

\begin{equation}\label{eq:linear_space_cardinality}
  \left|\mathcal{V}\right|\geq(2/\sqrt{e})^{d/2}
\end{equation}

We now define a set of functions $\mathcal{G}(\delta,h)$ parameterized by $\alpha\in\mathcal{V}$ with $\delta\in\mathbb{R}$ and $h_i:\mathbb{R}\rightarrow\mathbb{R}$ such that $g_\alpha\in\mathcal{G}(\delta,h)$ is defined as

$$g_\alpha(x)=\frac{\delta}{d}\sum_{i=1}^d\alpha_ih_i(x(i))=\frac{\delta}{d}\langle\alpha,H(x)\rangle$$

where $x(i)$ is the $i^{th}$ coordinate of $x\in\mathbb{R}^d$ and $H(x)=[h_1(x(1)),h_2(x(2)),\ldots,h_d(x(d))]^T$. Note that $\mathcal{G}\subset\mathcal{F}$ in this case, and the $h_i$ are left as arbitrary functions for now to allow different options to potentially yield different lower bounds. However, for this proof, they will be linear functions.

Additionally, we restrict the function $h$ by requiring that it satisfies Fubini's Theorem. That is, we require the following constraint. 

$$\frac{\delta}{d}\int_{x\in R}\sum_{i=1}^d\left|\alpha_i h_i(x(i))\right|<\infty$$

Therefore, we can apply Fubini's Theorem to get the following equality.

$$\int_{x\in R} g_\alpha(x)=\frac{\delta}{d}\int_{x\in R}\sum_{i=1}^d\alpha_i h(x(i))= \frac{\delta}{d}\sum_{i=1}^d\int_{x\in R}\alpha_i h_i(x(i))$$

\subsection{Minimum Distance between Functions in the Class}

We then let $\psi(\mathcal{G}(\delta,h))$ denote the discrepancy in the absolute value between the integral of any two distinct functions in $\mathcal{G}(\delta,h)$ over the region $R$. Let $\alpha\ne\beta$ for $\alpha,\beta\in\mathcal{V}$. We define the discrepancy as

$$\psi(\mathcal{G}(\delta,h)):=\inf_{\alpha,\beta\in\mathcal{V}}\inf_R\left|\int_{x\in R} g_\alpha(x)-\int_{x\in R} g_\beta(x)\right|$$

Then note that

\begin{equation*}\begin{split}
  \psi(\mathcal{G}(\delta,h)) 
    &= \inf_{\alpha,\beta\in\mathcal{V}}\inf_R\frac{\delta}{d}\sum_{i=1}^d\left|(\alpha_i-\beta_i)\int_{x\in R} h_i(x(i))\right|\\
    &\geq \inf_{\alpha,\beta\in\mathcal{V}}\inf_R\frac{2\delta}{d}\sum_{i=1}^d 1[\alpha_i\ne\beta_i]\inf_i\left|\int_{x\in R} h_i(x(i))\right|\\
    &\geq \frac{\delta}{2}\inf_R\inf_i\left|\int_{x\in R} h_i(x(i))\right|\\
\end{split}\end{equation*}

Thus
\begin{equation}\label{eq:disc_bound}
  \psi(\mathcal{G}(\delta,h))\geq \frac{\delta}{2}\inf_R\inf_i\left|\int_{x\in R} h_i(x(i))\right|
\end{equation}

Therefore, for any $\alpha,\beta\in\mathcal{V}$ with $\alpha\ne\beta$, we have 

$$\left|\int_{x\in R} g_\alpha(x)-\int_{x\in R} g_\beta(x)\right|\geq\frac{\delta}{2}\inf_R\inf_i\left|\int_{x\in R} h_i(x(i))\right|$$

Now, by using this bound, we show that, for any real value $I$, there can exist at most one function $g_\alpha$ such that $\int_R g_\alpha(x)$ is contained within an $l_1$ ball with radius equal to $1/3$ of the discrepancy.

\begin{lemma}\label{lemma:unique_disc_bound}
If
  \begin{equation}\label{eq:inf_lower_bound}
    \inf_R\inf_i\left|\int_{x\in R} h_i(x(i))\right|>0
  \end{equation}
then, for any valid region $R$, and $I=\int_{x\in R}f(x)$, there can be at most one $\alpha\in\mathcal{V}$ such that
  \begin{equation}\label{eq:unique_disc_bound}
    \left|I-\int_{x\in R} g_\alpha(x)\right|\leq\frac{1}{3}\psi(\mathcal{G}(\delta,h))
  \end{equation}
\end{lemma}

For a proof of this lemma, see section \ref{sec:lower-lemmas} in the appendix.

\subsection{Upper Bounding Probability of Estimator being Wrong}

Next, if the assumption in Lemma \ref{lemma:unique_disc_bound} holds, then we can claim that, if a model $\mathcal{M}$ can achieve a minimax error bounded as

\begin{equation}\label{eq:model_error_bound}
  \mathbb{E}_\phi\left[\epsilon(\mathcal{M},\mathcal{G}(\delta,h),\phi)\right]\leq\frac{1}{9}\psi(\mathcal{G}(\delta,h))
\end{equation}

then that model $\mathcal{M}$ can output a value $\hat{\alpha}(\mathcal{M})$ to be the $\alpha\in\mathcal{V}$ where $\left|I-\int_{x\in R} g_\alpha(x)\right|\leq\frac{1}{3}\psi(\mathcal{G}(\delta,h))$ if such an $\alpha$ exists, and if no such $\alpha$ exists, then the model chooses uniformly at random from $\mathcal{V}$. Note that Lemma \ref{lemma:unique_disc_bound} ensures that either one or zero such $\alpha$'s exists. Then we can now use Markov's inequality to prove that such an output from the model is wrong at most $\frac{1}{3}$ of the time. 

\begin{lemma}\label{lemma:alpha_estimator_upper_bound}
  If the assumptions in Lemma \ref{lemma:unique_disc_bound} hold, then if a model $\mathcal{M}$ satisfies inequality \eqref{eq:model_error_bound}, it can construct an estimator $\hat{\alpha}(\mathcal{M})$ to estimate the true $\alpha$ with an error upper bounded as
  
  $$\max_{\alpha^*\in\mathcal{V}}\textbf{P}_\phi\left[\hat{\alpha}(\mathcal{M})\ne\alpha^*\right]\leq\frac{1}{3}$$
  
  which implies that the model will only fail to retrieve the correct $g_\alpha$ with probability at most $\frac{1}{3}$.
\end{lemma}

For a proof of this lemma, see section \ref{sec:lower-lemmas} in the appendix.

\subsection{Defining an Oracle}

In this section, we define a specific oracle to be considered with our function class, inspired from \cite{Agarwal00} in the context of convex optimization. Let the oracle $\phi$ be defined in the following way. When a point $x_t$ is queried, the oracle chooses a dimension $i\in\{1,\ldots,d\}$ uniformly at random and generates $b$ from a Bernoulli distribution with parameter $p=1/2+\alpha_i\delta$. It then returns the following value. 

$$\phi(x_t,g_\alpha)=b\cdot \frac{h_i(x_t(i))}{2}+(1-b)\cdot\frac{-h_i(x_t(i))}{2}$$

Therefore, the expectation of the oracle on a function $g_\alpha\in\mathcal{G}(\delta,h)$ is defined as follows.

\begin{align*}
    \mathbb{E}&\left[\phi(x_t,g_\alpha)\right]\\
    &= \frac{1}{d}\sum_{i=1}^d(1/2+\alpha_i\delta)\frac{h_i(x_t(i))}{2}+(1/2-\alpha_i\delta)\frac{-h_i(x_t(i))}{2}\\
    &= \frac{\delta}{d}\sum_{i=1}^d\alpha_i h_i(x_t(i)) = g_\alpha(x_t)
\end{align*}

Thus we can conclude that the oracle is, in fact, unbiased. Next, we observe the uncentered second-order moment. 

\begin{align*}
    \mathbb{E}&\left[\phi(x_t,g_\alpha)^2\right]\\
    &= \frac{1}{d}\sum_{i=1}^d(1/2+\alpha_i\delta)\frac{h_i(x_t(i))^2}{4}+(1/2-\alpha_i\delta)\frac{h_i(x_t(i))^2}{4}\\
    &= \frac{1}{4d}\sum_{i=1}^d h_i(x_t(i))^2 = \frac{1}{4d}||H(x_t)||_2^2
\end{align*}

Then the formula for variance yields the following.

\begin{align*}
    \text{Var}(\phi(x_t,g_\alpha)) &= \mathbb{E}\left[\phi(x_t,g_\alpha)^2\right] - \mathbb{E}\left[\phi(x_t,g_\alpha)\right]^2\\
    &\leq \frac{1}{4d}||H(x_t)||_2^2
\end{align*}

Now let $S$ be the $l_\infty$ ball of radius $r$ centered at the origin. Then the variance of the oracle is upper bounded as $\text{Var}(\phi(x,g_\alpha))\leq\frac{1}{4d}||H(x)||_2^2\leq\frac{1}{4}\sup_{i\in\{1,\ldots,d\}}\sup_{z\in[-r,r]}h_i(z)^2$, so since $\text{Var}(\phi)\leq\sigma^2$ must hold, we have that $\sup_{i\in\{1,\ldots,d\}}\sup_{z\in[-r,r]}|h_i(z)|\leq 2\sigma$.

Using this oracle, we can now find an upper bound on its KL divergence. To define some notation, we will let $i_t$ denote the dimension the oracle chooses for $x_t$, and let $b_t$ denote the value of $b$ the oracle chooses for $x_t$. As such, the information revealed by the oracle is fully characterized by $\{(i_1,b_1),(i_2,b_2),\ldots,(i_T,b_T)\}$.

\subsection{Upper Bounding KL Divergence}

We now denote the information revealed by the oracle as $\mathcal{P}_\alpha^T$ and the distribution for a single $t$ as $\mathcal{P}_\alpha$. Note that, since $i$ is chosen uniformly at random, then $\mathcal{P}_\alpha(i,b)=\frac{1}{d}\mathcal{P}_{\alpha_i}(b)$. Then we can find an upper bound on the KL divergence between $\mathcal{P}_\alpha^T$ and $\mathcal{P}_{\alpha'}^T$ for $\alpha\ne\alpha'$ as follows.

\begin{align*}
    \text{KL}(\mathcal{P}_\alpha^T||\mathcal{P}_{\alpha'}^T) &= \sum_{t=1}^T \text{KL}(\mathcal{P}_\alpha(i_t,b_t)||\mathcal{P}_{\alpha'}(i_t,b_t))\\
    &= \sum_{t=1}^T\sum_{j=1}^d\frac{1}{d} \text{KL}(\mathcal{P}_{\alpha_j}(b_t)||\mathcal{P}_{\alpha'_j}(b_t))
\end{align*}

However, each term $\text{KL}(\mathcal{P}_{\alpha_j}(b_t)||\mathcal{P}_{\alpha'_j}(b_t))$ is at most the KL divergence between two Bernoulli distributions with parameters $1/2+\delta$ and $1/2-\delta$ respectively, which is upper bounded in the following way.

\begin{align*}
    \text{KL}&(\mathcal{P}_{\alpha_j}(b_t)||\mathcal{P}_{\alpha'_j}(b_t))\\ 
    &= \left(\frac{1}{2}+\delta\right)\log\left(\frac{1/2+\delta}{1/2-\delta}\right) + \left(\frac{1}{2}-\delta\right)\log\left(\frac{1/2-\delta}{1/2+\delta}\right)\\
    &= 2\delta\log\left(1+\frac{4\delta}{1-2\delta}\right) \leq \frac{2\delta\cdot 4\delta}{1-2\delta} = \frac{8\delta^2}{1-2\delta}
\end{align*}

Finally, we have $\frac{8\delta^2}{1-2\delta}\leq 16\delta^2$ when $0<\delta\leq 1/4$. Therefore, if $0<\delta\leq1/4$, then 

\begin{equation}\label{eq:KL_bound}
\text{KL}(\mathcal{P}_\alpha^T||\mathcal{P}_\beta^T)\leq 16T\delta^2
\end{equation}

\subsection{Lower Bounding Probability of Estimator being Wrong}

Suppose that a vector $\alpha^*$ is chosen uniformly at random from $\mathcal{V}$. Then let $\mathcal{M}$ be any model in $\mathbb{M}_T$, so $\mathcal{M}$ makes $T$ queries to the oracle $\phi$. Then we show that, if $\delta\leq\frac{1}{4}$, we can apply Fano's inequality \cite{Yu97}.

\begin{lemma}\label{lemma:alpha_estimator_lower_bound}
  Any model $\mathcal{M}$ that constructs any estimator $\hat{\alpha}(\mathcal{M})$ to estimate the true vertex $\alpha\in\mathcal{V}$ from T queries attains an error which is lower bounded as
  
  $$\max_{\alpha^*\in\mathcal{V}}\textbf{P}_\phi\left[\hat{\alpha}(\mathcal{M})\ne\alpha^*\right]\geq \left\{1-\frac{16T\delta^2+\log 2}{\frac{d}{2}\log(2/\sqrt{e})}\right\}$$
\end{lemma}

For a proof of this lemma, see section \ref{sec:lower-lemmas} in the appendix.

\subsection{Concluding the Proof of Theorem 1}

We now set $h_i(z)=z+r$. Then for some $\delta>0$, we analyze the set $\mathcal{G}(\delta,h)$. Since we require the region $R$ contain an $l_\infty$ ball of radius $r$, we can make sure our conditions hold by observing that 

\begin{align*}
\inf_R\inf_i\left|\int_{x\in R} h_i(x(i))\right| &= \left|(2r)^{d-1}\int_{-r}^{r}(z+r) dz\right|\\
&= \left|(2r)^{d-1}\left(\frac{z^2}{2}+rz\right)|_{-r}^{r}\right|\\ 
&= 2^dr^{d+1} > 0
\end{align*}

Thus, the conditions for Lemma \ref{lemma:unique_disc_bound} and Lemma \ref{lemma:alpha_estimator_upper_bound} hold. Now, let $k=\inf_R\inf_i\left|\int_{x\in R} h_i(x(i))\right|=2^dr^{d+1}$ and $\epsilon=k\delta/18$. Then if a model $\mathcal{M}$ achieves 

$$\mathbb{E}_\phi\left[\epsilon(\mathcal{M},\mathcal{G}(\delta,h),\phi)\right]\leq\frac{k\delta}{18}=\epsilon$$

we can use Lemma \ref{lemma:alpha_estimator_upper_bound} to get $\max_{\alpha^*\in\mathcal{V}}\textbf{P}\left[\hat{\alpha}(\mathcal{M})\ne\alpha^*\right]\leq\frac{1}{3}$. Likewise, from Lemma \ref{lemma:alpha_estimator_lower_bound}, we have that

$$\max_{\alpha^*\in\mathcal{V}}\textbf{P}_\phi\left[\hat{\alpha}(\mathcal{M})\ne\alpha^*\right]\geq \left\{1-\frac{16T(18\epsilon/k)^2+\log 2}{\frac{d}{2}\log(2/\sqrt{e})}\right\}$$

Therefore, we can combine the two terms to get

\begin{equation*}\begin{split}
  &\frac{1}{3}\geq\left\{1-\frac{16T(18\epsilon/k)^2+\log 2}{\frac{d}{2}\log(2/\sqrt{e})}\right\}\\
  \Rightarrow& \epsilon\geq k\sqrt{\frac{d\log(2/\sqrt{e})-3\log 2}{324\cdot 3\cdot 16T}} = 2^d r^{d+1}\sqrt{\frac{c_1 d-c_2}{T}}\\
  \Rightarrow& T=\Omega\left(\frac{d4^d r^{2d+2}}{\epsilon^2}\right)\text{ and }\epsilon\geq c 2^d r^{d+1}\sqrt{\frac{d}{T}}
\end{split}\end{equation*}

Therefore, we have proven Theorem \ref{theorem:general_lower_bound} since we conclude that $\epsilon^*(\mathcal{F},\phi)\geq c2^dr^{d+1}\sqrt{d/T}$.

\section{Concluding Remarks}

There are two primary ways to extend these results. First, other function classes can be considered which may yield better lower bounds than the linear function class used above. Additionally, other integration methods aside from Gaussian Quadrature and Simpson's Rule can be analyzed to determine their sample complexity upper bounds. Taking such steps could lead to finding methods with tight convergence rates for functions where Gaussian Quadrature cannot perform optimally, such as polynomials of degree four or greater.

\bibliographystyle{plain}
\bibliography{references}

\clearpage
\appendix
\onecolumn
\section{Proof of Lemmas Used in Theorem 1: Information-Theoretic Lower Bound}
\label{sec:lower-lemmas} 

In this section, we provide proofs for the lemmas used in Theorem \ref{theorem:general_lower_bound}. These lemmas consist of the uniqueness of integral values of functions in the linear function class, as well as the upper and lower bounds of estimating the vector from the integral value.

\subsection{Proof of Lemma \ref{lemma:unique_disc_bound}} 

\begin{proof}
First, we prove Lemma \ref{lemma:unique_disc_bound}, which states that there can exist at most one function from the linear function class which lies within a distance about any real value $I$ equal to 1/3 of the discrepancy. 

Assume there exists some $I\in\mathbb{R}$, a valid region $R$, and $\alpha,\beta\in\mathcal{V}$ where $\alpha\ne\beta$ such that the following inequality holds for both $\alpha$ and $\beta$.
$$\left|I-\int_{x\in R} g_\alpha(x)\right|\leq\frac{1}{3}\psi(\mathcal{G}(\delta,h))$$

Then
  \begin{equation*}\begin{split}
    \psi(\mathcal{G}(\delta,h)) &= \inf_{a,b\in\mathcal{V},a\ne b}\inf_{R'}\left|\int_{x\in R'} g_a(x)-\int_{x\in R'} g_b(x)\right|\\
    &\leq \left|\int_{x\in R} g_\alpha(x)-\int_{x\in R} g_\beta(x)\right|\\
    &= \left|\int_{x\in R} g_\alpha(x)-I+I-\int_{x\in R} g_\beta(x)\right|\\
    &\leq \left|I-\int_{x\in R} g_\alpha(x)\right|+\left|I-\int_{x\in R} g_\beta(x)\right|\\
    &\leq \frac{1}{3}\psi(\mathcal{G}(\delta,h))+\frac{1}{3}\psi(\mathcal{G}(\delta,h))\\
    &= \frac{2}{3}\psi(\mathcal{G}(\delta,h))
  \end{split}\end{equation*}
  
  However, $\psi(\mathcal{G}(\delta,h))\leq\frac{2}{3}\psi(\mathcal{G}(\delta,h))$ only holds if $\psi(\mathcal{G}(\delta,h))\leq 0$, and by equations \eqref{eq:disc_bound} and \eqref{eq:inf_lower_bound}, we have that $\psi(\mathcal{G}(\delta,h))>0$. Therefore, we have reached a contradiction, which implies that only one such $\alpha$ can exist.
\end{proof}
  
\subsection{Proof of Lemma \ref{lemma:alpha_estimator_upper_bound}} 

\begin{proof}
Now, we prove Lemma \ref{lemma:alpha_estimator_upper_bound}, which gives an upper bound on the probability that a model incorrectly predicts the true vector $\alpha^*$ from the vector set $\mathcal{V}$.

From Lemma \ref{lemma:unique_disc_bound}, at most one $\alpha$ can exist which satisfies inequality \eqref{eq:unique_disc_bound}, in which case the model chooses that $\alpha$. This implies that the model can be incorrect when the output $I$ from the model does not satisfy inequality \eqref{eq:unique_disc_bound} for that $\alpha$. Therefore, we get the following bound on the probability of the estimator being wrong.

$$\textbf{P}_\phi\left[\hat{\alpha}(\mathcal{M})\ne\alpha\right]\leq\textbf{P}_\phi\left[\epsilon(\mathcal{M},\mathcal{G}(\delta,h),\phi)\geq\frac{1}{3}\psi(\mathcal{G}(\delta,h))\right]$$

By applying Markov's inequality and using the bound in inequality \eqref{eq:model_error_bound}, we get

\begin{align*}
\textbf{P}_\phi\left[\epsilon(\mathcal{M},\mathcal{G}(\delta,h),\phi)\geq\frac{1}{3}\psi(\mathcal{G}(\delta,h))\right] &\leq \frac{\mathbb{E}\left[\epsilon(\mathcal{M},\mathcal{G}(\delta,h),\phi)\right]}{\frac{1}{3}\psi(\mathcal{G}(\delta,h))}\\
&\leq \frac{\frac{1}{9}\psi(\mathcal{G}(\delta,h))}{\frac{1}{3}\psi(\mathcal{G}(\delta,h))}\\
&= \frac{1}{3}
\end{align*}

So, since this holds for arbitrary $\alpha^*$, we can take the maximum over the $\alpha^*\in\mathcal{V}$ to prove Lemma \ref{lemma:alpha_estimator_upper_bound}.
\end{proof}

\subsection{Proof of Lemma \ref{lemma:alpha_estimator_lower_bound}} 

\begin{proof}
Finally, we prove Lemma \ref{lemma:alpha_estimator_lower_bound}, which then gives a lower bound on the probability of a model incorrectly predicting the true vector.

Using Fano's inequality, inequality \eqref{eq:linear_space_cardinality}, and inequality \eqref{eq:KL_bound}, we get the following bound, thus proving the lemma.
  \begin{equation*}\begin{split}
    \max_{\alpha^*\in\mathcal{V}}\textbf{P}&\left[\hat{\alpha}(\mathcal{M})\ne\alpha^*\right]\\
    &\geq \left\{1-\frac{\max_{\alpha,\beta\in\mathcal{V}}\left\{\text{KL}(\mathcal{P}_\alpha^T||\mathcal{P}_\beta^T)\right\}+\log 2}{\log\left|\mathcal{V}\right|}\right\}\\
    &\geq \left\{1-\frac{16T\delta^2+\log 2}{\log(2/\sqrt{e})^{d/2}}\right\}\\
    &= \left\{1-\frac{16T\delta^2+\log 2}{\frac{d}{2}\log(2/\sqrt{e})}\right\}\\
  \end{split}\end{equation*}
\end{proof}
\section{Proof of Theorem \ref{theorem:gq_upper_bound}: Sample Complexity Upper Bound of Gaussian Quadrature Method}
\label{sec:gq} 

In this section, we provide a proof for Theorem \ref{theorem:gq_upper_bound}, the sample complexity upper bound for the Gaussian Quadrature integral estimation method. Recall that, with this method, we integrate over a region $R=[-r,r]^d$, and we let $V_d=(\pm\frac{r}{\sqrt{3}}, \pm\frac{r}{\sqrt{3}},\ldots,\pm\frac{r}{\sqrt{3}})$ be the set of points at which this method will query the oracle function. Clearly, $\left|V_d\right|=2^d$.

Then let $T=m2^d$ be the total number of times the Gaussian Quadrature method will query the oracle function. More specifically, the method will query the oracle function $m$ times for each point $v\in V_d$ and will take the average of these values. Since the oracle function gives a noisy, but unbiased, estimate of the function, querying an individual point multiple times will reduce the variance the oracle function imposes on the integral estimation. 

Now, when computing the estimation of the integral, let $\hat{f}(x)=\frac{1}{m}\sum_{i=1}^m\phi(x,f)$. Then the Gaussian Quadrature method uses the following formula to estimate the value of the integral.

$$\int_{x\in R} f(x)=\int_{-r}^r\int_{-r}^r\ldots\int_{-r}^r f(x_1,x_2,\ldots,x_d)dx_1 dx_2\ldots dx_d\approx r^d\sum_{v\in V_d} \hat{f}(v)$$

\subsection{Proof of Lemma \ref{lemma:gq_exact}}

\begin{proof}
We first prove Lemma \ref{lemma:gq_exact} by showing that this multi-dimensional extension of Gaussian Quadrature is exact for polynomials of degree up to 3. It is already known that Gaussian Quadrature is exact in this manner for one dimension. Therefore, we can use induction to extend its exactness to higher dimensions. This will allow us to determine an error term on the integral estimation for functions that cannot be exactly estimated by this method. 

It is known that Gaussian Quadrature is exact in the one dimensional case. Therefore, for $d=1$, we have
\begin{align*}
\int_{-r}^r f(x)dx &= r\left(f\left(\frac{r}{\sqrt{3}}\right)+f\left(-\frac{r}{\sqrt{3}}\right)\right)\\
&= r\sum_{v\in V_1}f(v)
\end{align*}
where $V_1=(\pm\frac{r}{\sqrt{3}})$ as defined at the start of the section.

We now treat this as the base case for induction. Then assume that the claim holds for $d=n-1$. 

$$\int_{-r}^r\int_{-r}^r\ldots\int_{-r}^r f(x_1,\ldots,x_{n-1},x_n) dx_1\ldots dx_{n-1} = r^{n-1}\sum_{v\in V_{n-1}} f(v,x_n)$$

Note that, in this expression, $x_n$ is being held constant, and $V_{n-1}=(\pm\frac{r}{\sqrt{3}},\ldots,\pm\frac{r}{\sqrt{3}})$ with $\left|V\right|=2^{n-1}$. Now, for $d=n$, we can replace the innermost $n-1$ integrals using the above formula.

$$\int_{-r}^r\ldots\int_{-r}^r f(x_1,\ldots,x_n)dx_1\ldots dx_n = \int_{-r}^r r^{n-1}\sum_{v\in V_{n-1}}f(v,x_n)dx_n$$

However, this is now a one-dimensional integral in terms of only $x_n$. Therefore, we can now apply Gaussian Quadrature again to get the following formula.

\begin{align*}
  \int_{-r}^r r^{n-1}\sum_{v\in V_{n-1}}f(v,x_n)dx_n &= r\left(r^{n-1}\sum_{v\in V_{n-1}}f\left(v,\frac{r}{\sqrt{3}}\right)+r^{n-1}\sum_{v\in V_{n-1}}f\left(v,-\frac{r}{\sqrt{3}}\right)\right)\\
  &= r^n \sum_{v\in V_n} f(v)
\end{align*}

Therefore, by induction, we have that, for any $d$, Gaussian Quadrature is exact for polynomials of degree at most 3, so we get the following formula, which proves Lemma \ref{lemma:gq_exact}. 

$$\int_{x\in R} f(x)=\int_{-r}^r\int_{-r}^r\ldots\int_{-r}^r f(x_1,x_2,\ldots,x_d)dx_1 dx_2\ldots dx_d = r^d\sum_{v\in V_d} f(v)$$
\end{proof}

\subsection{Proof of Lemma \ref{lemma:gq_error_term}}

\begin{proof}
Next, because Gaussian Quadrature is exact for polynomials of degree up to 3, we can now use the error formula for Hermite Interpolation to find a formula for the error of the integral estimation given by Gaussian Quadrature without the noisy oracle function, thus proving Lemma \ref{lemma:gq_error_term}. This formula will then be able to be used to find an upper bound on the information-theoretic error of the Gaussian Quadrature method when it only has access to noisy function values.

  By Lemma \ref{lemma:gq_exact}, Gaussian Quadrature is exact for polynomials of degree up to 3. Now recall that we let $V_d=(\pm\frac{r}{\sqrt{3}}, \pm\frac{r}{\sqrt{3}},\ldots,\pm\frac{r}{\sqrt{3}})$ be the set of points at which this method will query the oracle function. Then for any estimation $r^d\sum_{v\in V_d}f(v)$ given by the Gaussian Quadrature method, we have that $\int_{x\in R} p_3(x)=r^d\sum_{v\in V_d}f(v)$ for any polynomial $p_3$ of degree at most 3, such that $p_3(v)=f(v)$ for all $v\in V_d$. 
  
  Then we can represent the error of the Gaussian Quadrature estimation in terms of $p_3$ using the following equation.
  
  $$\int_{x\in R} f(x)-r^d\sum_{v\in V_d}f(v)=\int_{x\in R} f(x)-\int_{x\in R} p_3(x)=\int_{x\in R}\left(f(x)-p_3(x)\right)$$
  
  Next, let $i\in\{1,\ldots,d\}$. Then by holding the other $d-1$ dimensions constant, Hermite Interpolation gives the following error formula which holds for any $x_i\in[-r,r]$
  
  $$f_i(x_i)-p_{3,i}(x_i)=\frac{f_i^{(4)}(\xi(x))c}{4!}$$
  
  where $c$ is defined using the following formula.
  
  $$c=\int_{-r}^{r}\left((x-\frac{r}{\sqrt{3}})(x+\frac{r}{\sqrt{3}})\right)^2dx=\frac{8r^5}{45}$$
  
  Therefore, $c\in\left[0,\frac{8r^5}{45}\right]$
  
  However, since the other $d-1$ dimensions are being held constant, the above expression holds for any $x_j\in\mathbb{R}$ where $i\ne j$ as long as $x_i\in[-r,r]$ since the error formula for Hermite Interpolation only depends on $x_i$. Note that, by changing the constant values these other dimensions are being held to, the above expression can change since $f_i$ and $p_{3,i}$ are defined in terms of the constant values $x_j$ for $j\ne i$.
  
  Additionally, at any point $x=(x_1,\ldots,x_i,\ldots,x_d)\in[-r,r]^d$, we have that $f(x)=f_i(x_i)$ by simply holding the $d-1$ other dimensions constant to their values at this point. Now we let $p_3$ be defined by the polynomials constructed using Hermite Interpolation. Then, $p_3(x)=p_{3,i}(x)$, so since Gaussian Quadrature is exact for any polynomial of degree up to 3, we get the following formula for the absolute value of the error at any $x\in[-r,r]^d$.
  
  $$\left|f(x)-p_3(x)\right|=\left|f_i(x_i)-p_{3,i}(x_i)\right|=\left|\frac{f_i^{(4)}(\xi(x))c}{4!}\right|$$
  
  Therefore, by taking the supremum of this formula over all dimensions $i$ and over all points $x^*\in[-r,r]^d$, we can upper bound the absolute value of the error. Note that $c\in[0,\frac{8r^5}{45}]$, so since $c\geq 0$, we can factor it out of the absolute value.
  
  $$\left|f(x)-p_3(x)\right|\leq\sup_{i,x^*}\left|\frac{f_i^{(4)}(\xi(x^*))*c}{4!}\right|= \frac{c}{4!}\sup_{i,x^*}\left|f_i^{(4)}(\xi(x^*))\right|$$
  
  Now, since the absolute value of an integral of a function is less than or equal to the integral of the absolute value of the function, we can upper bound the error of the Gaussian Quadrature estimate to prove the theorem. Note that $x^*$ depends on the supremum, not the integral, so the supremum can be factored out of the integral. Additionally, the volume of the region of integration $[-r,r]^d$ is $2^d r^d$. Thus, we prove Lemma \ref{lemma:gq_error_term}.
  
  \begin{align*}
    \left|\int_{x\in R} f(x)-r^d\sum_{v\in V_d}f(v)\right| &= \left|\int_{x\in R}\left(f(x)-p_3(x)\right)\right|\\
    &\leq \int_{x\in R}\left|f(x)-p_3(x)\right|\\
    &\leq \int_{x\in R} \frac{c}{4!}\sup_{i,x^*}\left|f_i^{(4)}(\xi(x^*))\right|\\
    &= \frac{c}{4!}\sup_{i,x^*}\left|f_i^{(4)}(\xi(x^*))\right| \int_{x\in R} 1\\
    &= \frac{c2^d r^d}{4!}\sup_{i,x^*}\left|f_i^{(4)}(\xi(x^*))\right|
  \end{align*}
\end{proof}  

\subsection{Proof of Theorem \ref{theorem:gq_upper_bound}}

\begin{proof}
We can now use the error bound from Lemma \ref{lemma:gq_error_term} to find a sample complexity upper bound on the error of the estimate produced by the Gaussian Quadrature method. 

  Let $S=\int_{x\in R} f(x)-r^d\sum_{v\in V_d}f(v)$ be the error for a non-noisy estimation produced by the Gaussian Quadrature method. Then since $\mathbf{E}[\phi(x,f)]=f(x)$ because the oracle function is unbiased, we can use the linearity of expectation to get
  
  $$\mathbf{E}\left[\hat{f}(x)\right]=\mathbf{E}\left[\frac{1}{m}\sum_{i=1}^m\phi(x,f)\right]=\frac{1}{m}\sum_{i=1}^m\mathbf{E}[\phi(x,f)]=\mathbf{E}[\phi(x,f)]=f(x)$$
  
  Therefore, we can conclude that $\hat{f}(x)$, the average of the $m$ calls to the noisy oracle function, is also unbiased. Then we can use this fact to get the following expression.
  
  $$\mathbf{E}_\phi\left[r^d\sum_{v\in V_d}\hat{f}(v)+S\right]=\mathbf{E}_\phi\left[\int_{x\in R} f(x)\right]+\mathbf{E}_\phi\left[r^d\sum_{v\in V_d}\hat{f}(v)-r^d\sum_{v\in V_d}f(v)\right]=\int_{x\in R} f(x)+0=\int_{x\in R} f(x)$$
  
  Additionally, since each call to the oracle function is independent, we can consider the variance of $\hat{f}(x)$.
  
  \begin{align*}
    \text{Var}\left(\hat{f}(x)\right) &= \text{Var}\left(\frac{1}{m}\sum_{i=1}^m\phi(x,f)\right)\\
    &= \frac{1}{m^2}\text{Var}\left(\sum_{i=1}^m\phi(x,f)\right)\\
    &= \frac{1}{m^2}\sum_{i=1}^m\text{Var}\left(\phi(x,f)\right)\\
    &\leq \frac{1}{m^2}\sum_{i=1}^m \sigma^2\\
    &= \frac{m\sigma^2}{m^2}\\
    &= \frac{\sigma^2}{m}
  \end{align*}
  
  Therefore, we can find the variance of $r^d\sum_{v\in V_d}\hat{f}(v)+S$. 
  
  \begin{align*}
    \text{Var}\left(r^d\sum_{v\in V_d}\hat{f}(v)+S\right) &= r^{2d}\sum_{v\in V_d}\text{Var}\left(\hat{f}(v)\right)\\
    &\leq r^{2d}\sum_{v\in V_d}\frac{\sigma^2}{m}\\
    &= \frac{2^d r^{2d}\sigma^2}{m}
  \end{align*}
  
  Next, since $\mathbf{E}_\phi\left[r^d\sum_{v\in V_d}\hat{f}(v)+S\right]=\int_{x\in R} f(x)$, Chebyshev's Inequality gives the following bound.
  
  $$\mathbf{P}\left(\left|r^d\sum_{v\in V_d}\hat{f}(v)+S-\int_{x\in R} f(x)\right|>\epsilon\right)\leq \min\left(\frac{2^d r^{2d}\sigma^2}{m\epsilon^2},1\right)$$
  
  Now, let $X=\left|r^d\sum_{v\in V_d}\hat{f}(v)+S-\int_{x\in R} f(x)\right|$. Then the Layer Cake representation gives the following expression for $\mathbf{E}[X]$.
  
  \begin{align*}
    \mathbf{E}\left[X\right] &= \int_0^\infty \mathbf{P}\left(X>\alpha\right)d\alpha\\
    &\leq \int_0^\infty \min\left(\frac{2^d r^{2d}\sigma^2}{m\alpha^2},1\right)\\
    &\leq \int_0^{\sqrt{2^d r^{2d}\sigma^2/m}}1d\alpha+\frac{2^d r^{2d}\sigma^2}{m}\int_{\sqrt{2^d r^{2d}\sigma^2/m}}^\infty \frac{1}{\alpha^2} d\alpha\\
    &= \sqrt{\frac{2^d r^{2d}\sigma^2}{m}}+\frac{2^d r^{2d}\sigma^2}{m}\sqrt{\frac{m}{2^d r^{2d}\sigma^2}}\\
    &= 2\sqrt{\frac{2^d r^{2d}\sigma^2}{m}}\\
    &= \frac{2^{d/2+1}r^d\sigma}{\sqrt{m}}
  \end{align*}
  
  Next, since $\mathbf{E}_\phi\left[\left|r^d\sum_{v\in V_d}\hat{f}(v)+S-\int_{x\in R} f(x)\right|\right]=\mathbf{E}[X]\leq \frac{2^{d/2+1}r^d\sigma}{\sqrt{m}}$, we can get the following upper bound.
  
  $$\mathbf{E}_\phi\left[\left|r^d\sum_{v\in V_d}\hat{f}(v)-\int_{x\in R} f(x)\right|\right]\leq \frac{2^{d/2+1}r^d\sigma}{\sqrt{m}}-S$$
  
  However, Lemma \ref{lemma:gq_error_term} tells us that $\left|S\right|\leq \frac{c2^d r^d}{4!}\sup_{i,x^*}\left|f_i^{(4)}(\xi(x^*))\right|$. Therefore, we can substitute this formula into the above inequality to get a new upper bound.
  
  $$\mathbf{E}_\phi\left[\left|r^d\sum_{v\in V_d}\hat{f}(v)-\int_{x\in R} f(x)\right|\right]\leq \frac{2^{d/2+1}r^d\sigma}{\sqrt{m}}+\frac{c2^d r^d}{4!}\sup_{i,x^*}\left|f_i^{(4)}(\xi(x^*))\right|$$
  
  Then since $\left|f_i^{(4)}(x)\right|\leq K$ for all $x\in[-r,r]^d$, and since $c\in[0,8r^5/45]$, we get the following inequality.
  
  $$\mathbf{E}_\phi\left[\left|r^d\sum_{v\in V_d}\hat{f}(v)-\int_{x\in R} f(x)\right|\right]\leq \frac{2^{d/2+1}r^d\sigma}{\sqrt{m}}+\frac{2^{d+1} r^{d+5}}{6\cdot 45}K$$
  
  However, since $T=m2^d$, we can then substitute $m=T/2^d$ into the above inequality to prove the first part of Theorem \ref{theorem:gq_upper_bound}.
  
  $$\epsilon(GQ,F,\phi)\leq \frac{2^{d+1}r^d\sigma}{\sqrt{T}}+\frac{2^{d+1} r^{d+5}}{6\cdot 45}K$$
  
  Finally, it is trivial to see that, if $K=0$, then the other bound in Theorem \ref{theorem:gq_upper_bound} holds.
  
  $$\epsilon(GQ,F,\phi)\leq \frac{2^{d+1}r^d\sigma}{\sqrt{T}}$$
\end{proof}
\section{Proof of Theorem \ref{theorem:gq_gaussian_oracle}: Gaussian Quadrature Error with Gaussian Oracle}
\label{sec:gaussianoracle} 

\begin{proof}
We now assume that the noisy oracle function follows a Gaussian distribution. Then instead of variance being upper bounded by $\sigma^2$, we have that $\text{Var}(\phi)=\sigma^2$. This removes the need for Chebyshev's inequality in the proof for the information theoretic upper bound, instead allowing us to find a tight formula. Note that the work for Lemmas \ref{lemma:gq_exact} and \ref{lemma:gq_error_term} still hold in this case.

  Let $S=\sum_{v\in V}f(v)-\int_{x\in R} f(x)$ be the error for a non-noisy estimation produced by the Gaussian Quadrature method. Then since $\mathbf{E}[\phi(x,f)]=f(x)$ because the Gaussian oracle function is still unbiased, we again get
  
  $$\mathbf{E}\left[\hat{f}(x)\right]=\mathbf{E}\left[\frac{1}{m}\sum_{i=1}^m\phi(x,f)\right]=\frac{1}{m}\sum_{i=1}^m\mathbf{E}[\phi(x,f)]=\mathbf{E}[\phi(x,f)]=f(x)$$
  
  Additionally, since each call to the oracle function relies on an independent Gaussian random variable with variance $\sigma^2$, we can consider the variance of $\hat{f}(x)$.
  
  \begin{align*}
    \text{Var}\left(\hat{f}(x)\right) &= \text{Var}\left(\frac{1}{m}\sum_{i=1}^m\phi(x,f)\right)\\
    &= \frac{1}{m^2}\sum_{i=1}^m\text{Var}\left(\phi(x,f)\right)\\
    &= \frac{1}{m^2}\sum_{i=1}^m \sigma^2\\
    &= \frac{\sigma^2}{m}
  \end{align*}
  
  In fact, since $\phi(x,f)$ is a Gaussian random variable with mean $f(x)$ and variance $\sigma^2$, we have that $\hat{f}(x)$ is a Gaussian random variable with mean $f(x)$ and variance $\sigma^2/m$. That is, we have
  
  $$\hat{f}(x)\sim \mathcal{N}(f(x),\sigma^2/m)$$
  
  Next, we can consider that $r^d\sum_{v\in V_d}\hat{f}(v)$ is a sum of these normal random variables. Therefore, since the sum of normal random variables is a normal random variable whose mean is the sum of the original means and whose variance is the sum of the original variances, we get 
  
  $$r^d\sum_{v\in V_d}\hat{f}(v)\sim\mathcal{N}\left(r^d\sum_{v\in V_d}f(v),r^{2d}\sum_{v\in V_d}\frac{\sigma^2}{m}\right)=\mathcal{N}\left(r^d\sum_{v\in V_d}f(v),\frac{2^d r^{2d}\sigma^2}{m}\right)$$
  
  Finally, since we wish to consider the expected error of the Gaussian Quadrature method, we now note that the expected error is a normal random variable whose mean is shifted by $\int_{x\in R} f(x)$.
  
  $$r^d\sum_{v\in V_d}\hat{f}(v)-\int_{x\in R} f(x)\sim\mathcal{N}\left(r^d\sum_{v\in V_d}f(v)-\int_{x\in R} f(x),\frac{2^d r^{2d}\sigma^2}{m}\right)=\mathcal{N}\left(S,\frac{2^d r^{2d}\sigma^2}{m}\right)$$
  
  Now, consider the case where $K=0$. Then by Lemma \ref{lemma:gq_exact}, the non-noisy Gaussian Quadrature estimate is exact. Therefore, the error for the non-noisy estimation, $S$, must be zero, so we apply the fact that, for a Gaussian variable $Z\sim\mathcal{N}(0,s^2)$, $\mathbf{E}[|Z|]=s\sqrt{2/\pi}$ to get the following formula.
  
  $$\mathbf{E}\left[\left|r^d\sum_{v\in V_d}\hat{f}(v)-\int_{x\in R} f(x)\right|\right]=\sqrt{\frac{2^d r^{2d}\sigma^2}{m}}\sqrt{\frac{2}{\pi}}= \frac{2^{d/2}r^d\sigma}{\sqrt{m}}\sqrt{\frac{2}{\pi}}$$
  
  Next, consider the case where $K=O(1/\sqrt{m})>0$. Then we look at the error for a class of functions for which the results for Lemma \ref{lemma:gq_error_term} hold with equality. That is, we assume the error for the non-noisy Gaussian Quadrature estimation is 
  
  $$S=\frac{c2^d r^d}{4!\sqrt{m}}$$.
  
  Then for a Gaussian variable $Z\sim\mathcal{N}(\mu,s^2)$, we know that $\mathbf{E}[|Z|]=s\exp\left(-\mu^2/(2s^2)\right)\sqrt{2/\pi}+\mu\text{ erf}\left(\mu/(\sqrt{2}s)\right)$ where erf is the Gauss error function. Thus we can represent the expectation using the following formula, and we can substitute the above formula for $S$ to get an explicit formula for the error.
  
  \begin{align*}
    \mathbf{E}_\phi\left[\left|r^d\sum_{v\in V_d}\hat{f}(v)-\int_{x\in R} f(x)\right|\right] &= \frac{2^{d/2}r^d\sigma}{\sqrt{m}}\exp\left(-\frac{S^2 m}{2^{d+1}r^{2d}\sigma^2}\right)\sqrt{\frac{2}{\pi}}+S\text{ erf}\left(\frac{S\sqrt{m}}{2^{d/2}r^d\sigma\sqrt{2}}\right)\\
    &= \frac{2^{d/2}r^d\sigma}{\sqrt{m}}\exp\left(-\frac{(c2^d r^d)^2 }{4!2^{d+1}r^{2d}\sigma^2}\right)\sqrt{\frac{2}{\pi}}+\frac{c2^d r^d}{4!\sqrt{m}}\text{ erf}\left(\frac{c2^d r^d}{4!2^{d/2}r^d\sigma\sqrt{2}}\right)\\
    &= \frac{2^{d/2}r^d\sigma}{\sqrt{m}}\exp\left(-\frac{c^2 2^d }{48\sigma^2}\right)\sqrt{\frac{2}{\pi}}+\frac{c2^d r^d}{4!\sqrt{m}}\text{ erf}\left(\frac{c2^{d/2}}{24\sigma\sqrt{2}}\right)\\
    &= \frac{2^{d/2}r^d\sigma}{\sqrt{m}}\exp\left(-\frac{c_1 2^d }{\sigma^2}\right)\sqrt{\frac{2}{\pi}}+\frac{c_2 2^d r^d}{\sqrt{m}}\text{ erf}\left(\frac{c_3 2^{d/2}}{\sigma}\right)
  \end{align*}
  
  Finally, we again note that $T=m2^d$, so we can substitute $m=T/2^d$ into the two above formulas. This yields the following two expressions, thus proving Theorem \ref{theorem:gq_gaussian_oracle}. First, for arbitrary $K$, we have the following formula from the first part of the theorem.
  
   $$\epsilon(GQ,F,\phi)= \frac{2^{d}r^d\sigma}{\sqrt{T}}\exp\left(-\frac{c_1 2^d }{\sigma^2}\right)\sqrt{\frac{2}{\pi}}+\frac{c_2 2^{3d/2} r^d}{\sqrt{T}}\text{erf}\left(\frac{c_3 2^{d/2}}{\sigma}\right)$$
  
  Then for $K=0$, we get the following formula from the second part of the theorem.
  
  $$\epsilon(GQ,F,\phi)=\frac{2^{d}r^d\sigma}{\sqrt{T}}\sqrt{\frac{2}{\pi}}$$
\end{proof}
\section{Experiments for Gaussian Quadrature with Gaussian Oracle}
\label{sec:gqexperiment}

In addition to the proof for Theorem \ref{theorem:gq_gaussian_oracle}, this result was also tested experimentally. The plots for these experiments are shown below. They compare the average error to the number of queries made, the number of dimensions for the input space, and the size of the integration region. Note that the Gaussian Quadrature method queries a total of $T=m2^d$ points, so the first plot consists of the average error of the Gaussian Quadrature method at integer multiples of $2^d$. 

For the first experiment, shown in Figure \ref{fig:error-vs-t}, 100 different cubic polynomials were randomly generated, and Gaussian Quadrature was used to estimate the integral of these polynomials over the region $[-5,5]^{10}$. The red curve demonstrates the expected error formulas as determined by Theorem \ref{theorem:gq_gaussian_oracle} while the blue curve demonstrates the experimental error averaged over these 100 polynomials. 

\begin{figure}[H]
    \begin{center}
        \includegraphics[scale = 0.7]{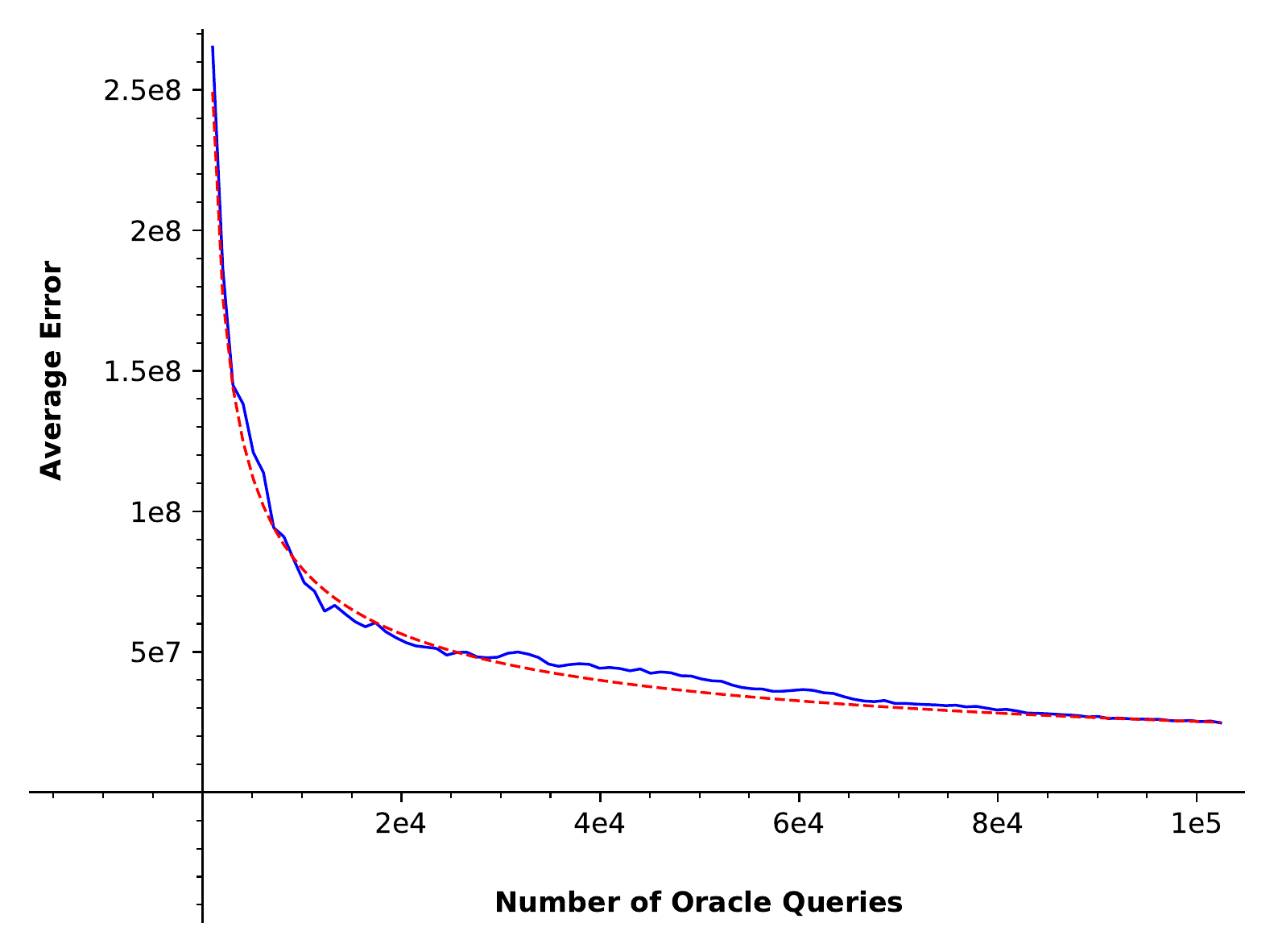}
        \caption{This figure shows the average error of the Gaussian Quadrature method versus the number of queries made to the oracle function. The blue curve demonstrates the experimental results, and the red curve plots the equation from Theorem \ref{theorem:gq_gaussian_oracle}, $\epsilon(GQ,F,\phi)=\frac{2^{d}r^d\sigma}{\sqrt{T}}\sqrt{\frac{2}{\pi}}=\frac{2^{10}5^{10}}{\sqrt{T}}\sqrt{\frac{2}{\pi}}\approx 7.9788456\cdot 10^9/\sqrt{T}$.}
        \label{fig:error-vs-t}
    \end{center}
\end{figure}

In the second experiment, shown in Figure \ref{fig:error-vs-d}, 100 cubic polynomials were generated for each value of $d$ ranging from 1 to 16. Then Gaussian Quadrature was used to estimate the integral of these polynomials over the region $[-5,5]^d$ by querying each point a total of 4 times. Thus, $T=4\cdot 2^d$ in this case. Again, the red curve demonstrates the expected error from Theorem \ref{theorem:gq_gaussian_oracle} while the blue curve demonstrates the experimental error averaged over the different polynomials. 

\begin{figure}[H]
    \begin{center}
        \includegraphics[scale = 0.7]{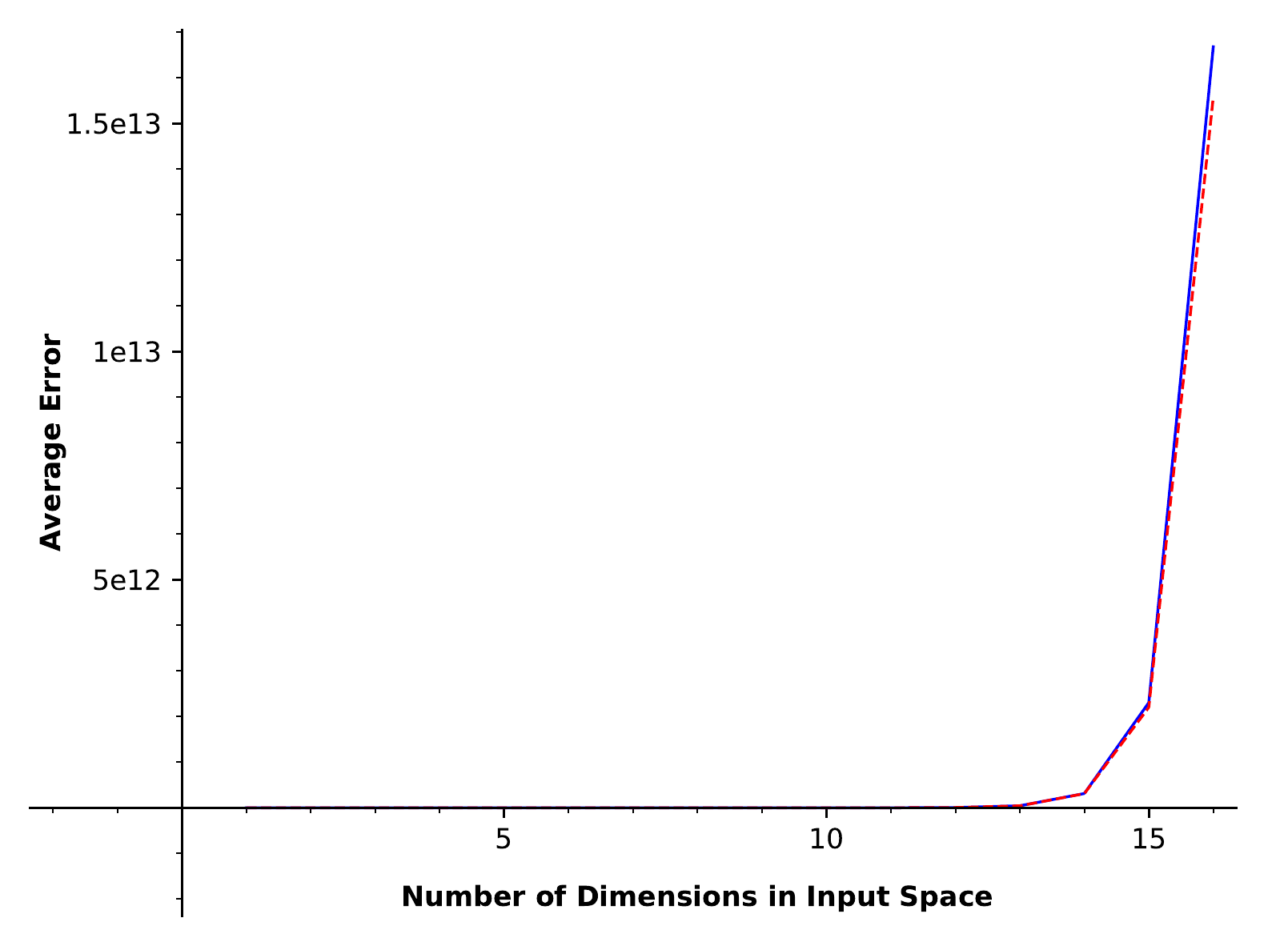}
        \caption{This figure shows the average error of the Gaussian Quadrature method versus the number of dimensions of the input space. The blue curve demonstrates the experimental results, and the red curve plots the equation from Theorem \ref{theorem:gq_gaussian_oracle}, $\epsilon(GQ,F,\phi)=\frac{2^{d}r^d\sigma}{\sqrt{T}}\sqrt{\frac{2}{\pi}}=\frac{2^{d}5^{d}}{\sqrt{4\cdot 2^d}}\sqrt{\frac{2}{\pi}}\approx 0.3989422804\cdot\sqrt{50}^d$.}
        \label{fig:error-vs-d}
    \end{center}
\end{figure}

Figure \ref{fig:error-vs-d-log} then displays the $log$ of the average error from the data shown above. Here, it can be more easily seen that the average error matches the formula from Theorem \ref{theorem:gq_gaussian_oracle} at the lower values of $d$.

\begin{figure}[H]
    \begin{center}
        \includegraphics[scale = 0.7]{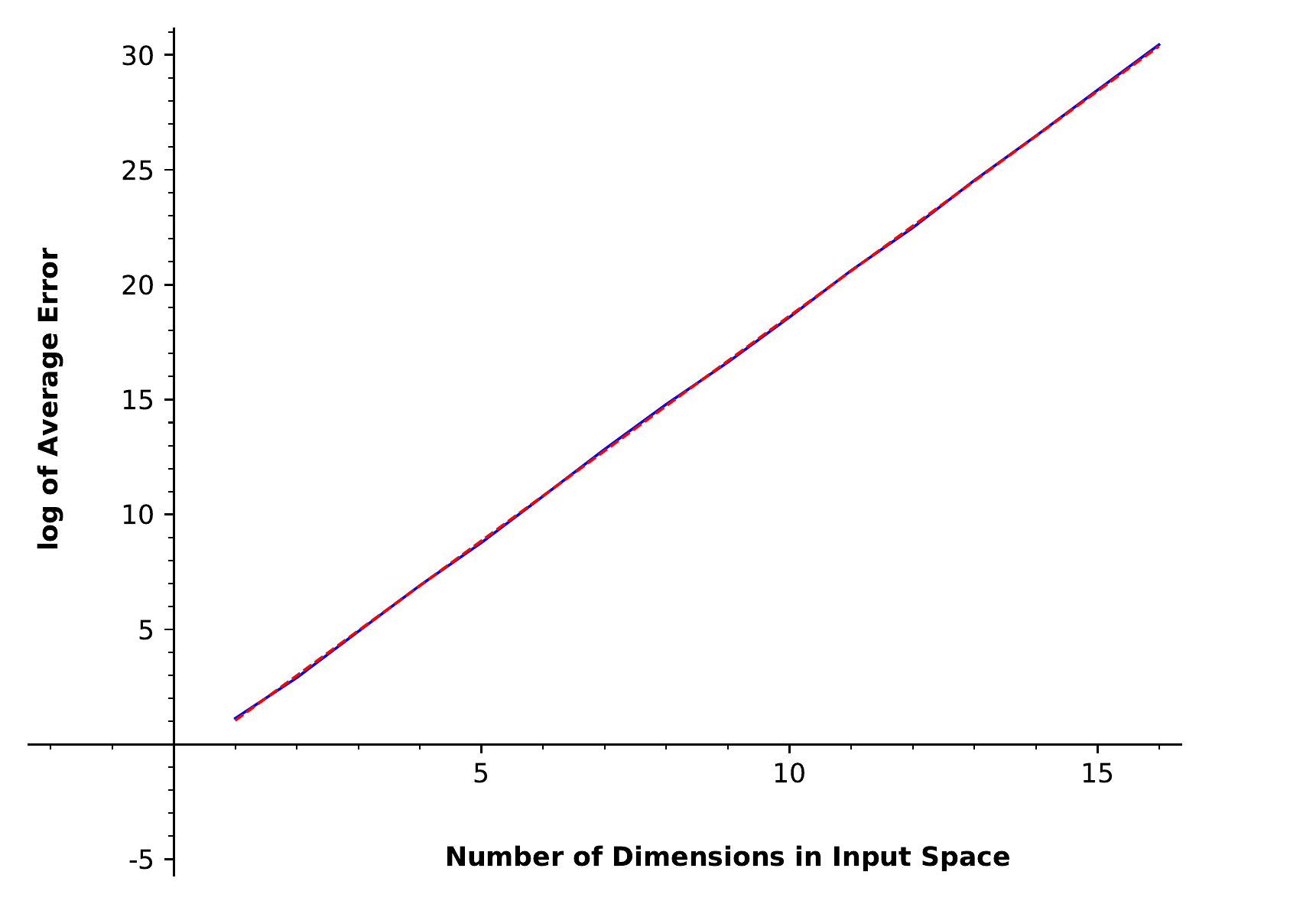}
        \caption{This figure shows the log of average error of the Gaussian Quadrature method versus the number of dimensions of the input space. The blue curve demonstrates the $\log$ of the experimental results, and the dashed red curve plots the $\log$ of the equation from Theorem \ref{theorem:gq_gaussian_oracle}.}
        \label{fig:error-vs-d-log}
    \end{center}
\end{figure}

In the third experiment, shown in Figure \ref{fig:error-vs-r}, 100 cubic polynomials were generated for each value of $r$, where $r$ consisted of integer powers of 2 ranging from 1/32 to 1024. Then Gaussian Quadrature was again used to estimate the integral of these polynomials over the region $[-r,r]^{10}$ by querying each point a total of 4 times. Thus, $T=4\cdot 2^{10}$ in this case. Again, the red curve demonstrates the expected error from Theorem \ref{theorem:gq_gaussian_oracle} while the blue curve demonstrates the experimental error averaged over the different polynomials. 

\begin{figure}[H]
    \begin{center}
        \includegraphics[scale = 0.7]{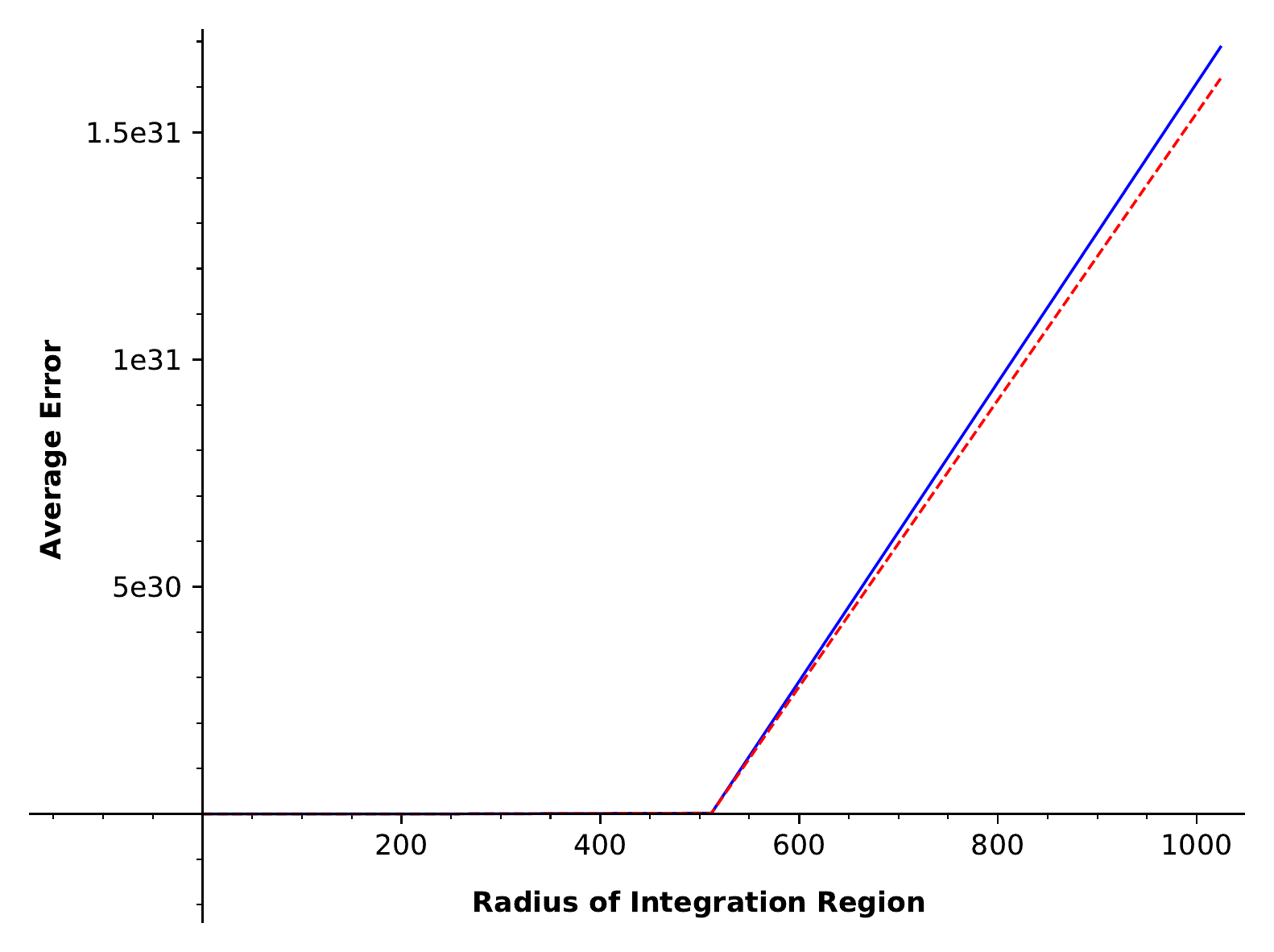}
        \caption{This figure shows the average error of the Gaussian Quadrature method versus the size of the integration region. The blue curve demonstrates the experimental results, and the red curve plots the equation from Theorem \ref{theorem:gq_gaussian_oracle}, $\epsilon(GQ,F,\phi)=\frac{2^{d}r^d\sigma}{\sqrt{T}}\sqrt{\frac{2}{\pi}}=\frac{2^{10}r^{10}}{\sqrt{4\cdot 2^{10}}}\sqrt{\frac{2}{\pi}}\approx 12.76615297\cdot r^{10}$.}
        \label{fig:error-vs-r}
    \end{center}
\end{figure}

Finally, in Figure \ref{fig:error-vs-r-log}, we display the $log_2$ of the average error with respect to the $log_2$ of $r$. That is, since $r$ consisted of powers of 2, we display the $log_2$ of the error versus the power of 2. Again, it is more clear that the average error matches the formula from Theorem \ref{theorem:gq_gaussian_oracle} at the smaller values of $r$. 

\begin{figure}[H]
    \begin{center}
        \includegraphics[scale = 0.7]{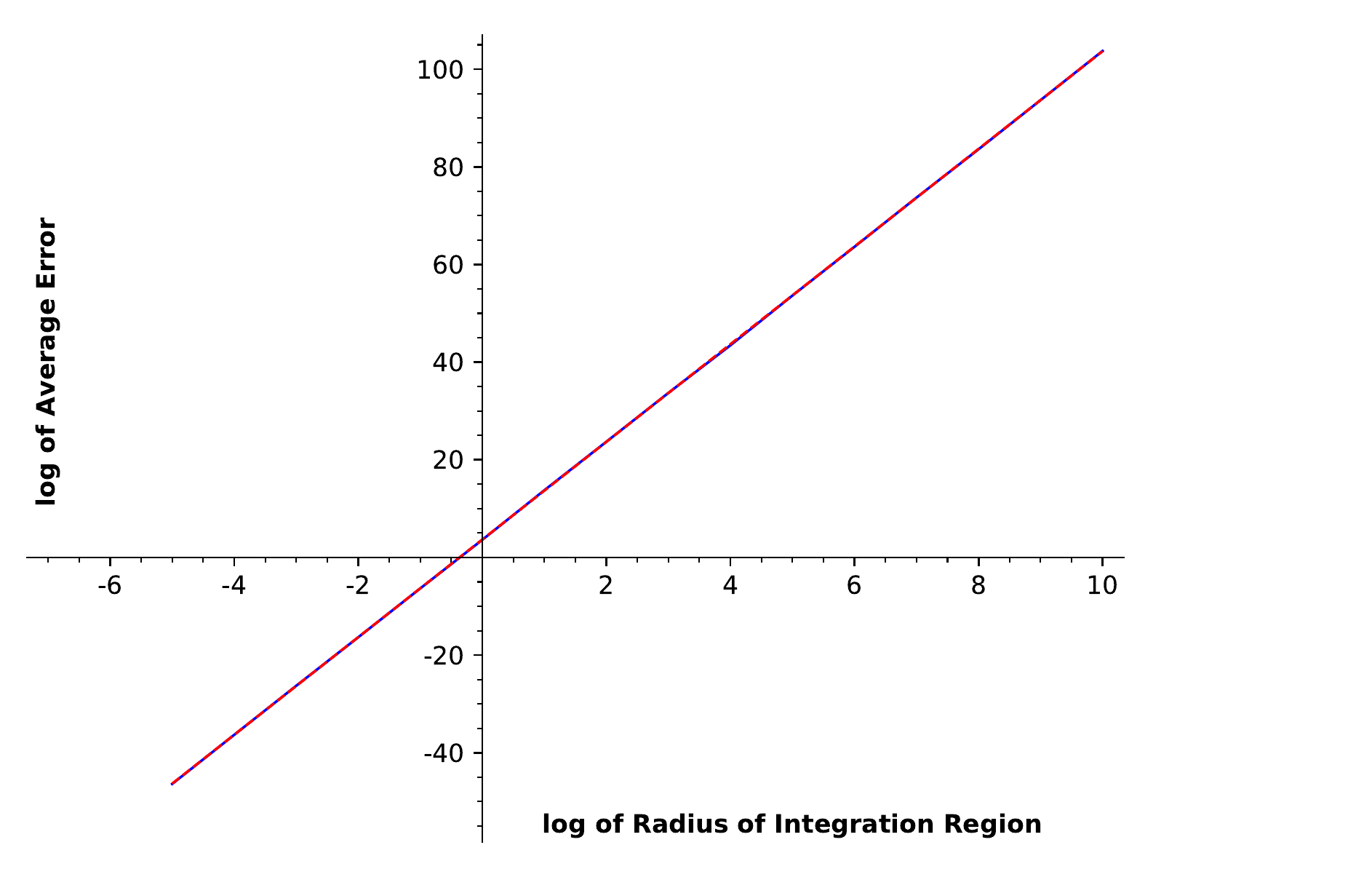}
        \caption{This figure shows the log of average error of the Gaussian Quadrature method versus the log of the size of the integration region. The blue curve demonstrates the $\log_2$ of the experimental results, and the dashed red curve plots the $\log_2$ of the equation from Theorem \ref{theorem:gq_gaussian_oracle}.}
        \label{fig:error-vs-r-log}
    \end{center}
\end{figure}
\section{Sample Complexity Upper Bound of Simpson's Rule Method}
\label{sec:simpsonsrule} 

We now consider the sample complexity upper bound for the Simpson's Rule integral estimation method. With this method, we integrate over a region $R=[a_1,b_1]\times[a_2,b_2]\times\ldots\times[a_d,b_d]$, and we let $V_d=\{(v_1,\ldots,v_d)\mid v_i\in\{a_i,\frac{a_i+b_i}{2},b_i\}\}$ be the set of points at which this method will query the oracle function. Clearly, $\left|V_d\right|=3^d$.

Then let $T=m3^d$ be the total number of times the Simpson's Rule method will query the oracle function. More specifically, the method will query the oracle function $m$ times for each point $v\in V_d$ and will take the average of these values. Since the oracle function gives a noisy, but unbiased, estimate of the function, querying an individual point multiple times will reduce the variance the oracle function imposes on the integral estimation.

Now, when computing the estimation of the integral, let $\hat{f}(x)=\frac{1}{m}\sum_{i=1}^m\phi(x,f)$. Next, we let $m_d(v)=\sum_{i=1}^d 1[v_i=\frac{a_i+b_i}{2}]$ be the count for how many components in a vector $v=(v_1,\ldots,v_d)\in V_d$ lie on the midpoint of the region in each dimension, and let $w_d(v)=4^{m_d(v)}\prod_{i=1}^d\frac{b_i-a_i}{6}$ be the weights corresponding to each point $v\in V_d$. Then the Simpson's Rule method uses the following formula to estimate the value of the integral.

$$\int_{x\in R} f(x)=\int_{a_d}^{b_d}\ldots\int_{a_2}^{b_2}\int_{a_1}^{b_1} f(x_1,x_2,\ldots,x_d)dx_1 dx_2\ldots dx_d\approx \sum_{v\in V_d} w_d(v)\hat{f}(v)=\prod_{i=1}^d\frac{b_i-a_i}{6}\sum_{v\in V_d}4^{m_d(v)}\hat{f}(v)$$

\subsection{Simpson's Rule is Exact for Degree 3 Polynomials}

We now show that this multi-dimensional extension of Simpson's Rule is exact for polynomials of degree up to 3. It is already known that Simpson's Rule is exact in this manner for one dimension. Therefore, we can use induction to extend its exactness to higher dimensions. This will allow us to determine an error term on the integral estimation for functions that cannot be exactly estimated by this method. 

\begin{lemma}\label{lemma:sr_exact}
If f is a polynomial of degree at most 3, then a non-noisy estimation from the Simpson's Rule method will exactly estimate the integral. That is,
  \begin{equation}\label{eq:sr_exact}
    \int_{x\in R} f(x)=\sum_{v\in V_d} w_d(v)f(v)
  \end{equation}
\end{lemma}

\begin{proof}
It is known that Simpson's Rule is exact in the one dimensional case. Therefore, for $d=1$, we have
\begin{align*}
\int_{a}^{b} f(x)dx &= \frac{b-a}{6}\left(f(a)+4f\left(\frac{a+b}{2}\right)+f(b)\right)\\
&= \sum_{v\in V_1}w_1(v)f(v)
\end{align*}
where $V_1=(a,\frac{a+b}{2},b)$ as defined at the start of the section.

We now treat this as the base case for induction. Then assume that the claim holds for $d=n-1$. 

$$\int_{a_{n-1}}^{b_{n-1}}\ldots\int_{a_1}^{b_1} f(x_1,\ldots,x_{n-1},x_n) dx_1\ldots dx_{n-1} = \sum_{v\in V_{n-1}} w_{n-1}(v)f(v,x_n)$$

Note that, in this expression, $x_n$ is being held constant, and $V_{n-1}=\{(v_1,\ldots,v_{n-1})\mid v_i\in\{a_i,\frac{a_i+b_i}{2},b_i\}\}$ with $\left|V\right|=3^{n-1}$. Now, for $d=n$, we can replace the innermost $n-1$ integrals using the above formula.

$$\int_{a_n}^{b_n}\int_{a_{n-1}}^{b_{n-1}}\ldots\int_{a_1}^{b_1} f(x_1,\ldots,x_n)dx_1\ldots dx_n = \int_{a_n}^{b_n} \sum_{v\in V_{n-1}}w_{n-1}(v)f(v,x_n)dx_n$$

However, this is now a one-dimensional integral in terms of only $x_n$. Therefore, we can now apply Simpson's Rule again to get the following formula.

\begin{align*}
  & \int_{a_n}^{b_n} \sum_{v\in V_{n-1}}w_{n-1}(v)f(v,x_n)dx_n\\ 
  &= \frac{b_n-a_n}{6}\left(\sum_{v\in V_{n-1}}w_{n-1}(v)f(v,a_n)+4\sum_{v\in V_{n-1}}w_{n-1}(v)f(v,\frac{a_n+b_n}{2})+\sum_{v\in V_{n-1}}w_{n-1}(v)f(v,b_n)\right)\\
  &= \sum_{v\in V_{n-1}}w_{n}(v,a_n)f(v,a_n)+\sum_{v\in V_{n-1}}w_{n}(v,\frac{a_n+b_n}{2})f(v,\frac{a_n+b_n}{2})+\sum_{v\in V_{n-1}}w_{n}(v,b_n)f(v,b_n)\\
  &= \sum_{v\in V_n} w_{n}(v)f(v)
\end{align*}

Therefore, by induction, we have that, for any $d$, Simpson's Rule is exact for polynomials of degree at most 3, so we get the following formula, which proves the theorem. 

$$\int_{x\in R} f(x)=\int_{a_d}^{b_d}\ldots\int_{a_1}^{b_1} f(x_1,x_2,\ldots,x_d)dx_1 dx_2\ldots dx_d = \sum_{v\in V_d} w_d(v)f(v)$$

\end{proof}

\subsection{Finding an Error Term for Simpson's Rule}

Next, because Simpson's Rule is exact for polynomials of degree up to 3, we can now use the error formula for Hermite Interpolation to find a formula for the error of the integral estimation given by Simpson's Rule without the noisy oracle function. This formula will then be able to be used to find an upper bound on the information-theoretic error of the Simpson's Rule method when it only has access to noisy function values.

\begin{lemma}\label{lemma:sr_error_term}

A non-noisy estimation from the Simpson's Rule method will achieve an error term with the following upper bound

  \begin{equation}\label{eq:sr_error_non_noisy}
    \left|\int_{x\in R} f(x)-\sum_{v\in V_d}w_d(v)f(v)\right|\leq \frac{c}{4!}\sup_{i,x^*}\left|f_i^{(4)}(\xi(x^*))\right|\prod_{i=1}^d(b_i-a_i)
  \end{equation}
  
where $c\in\left[0,\max_i\frac{(b_i-a_i)^7}{840}\right]$. In this bound, the supremum is considering the maximum fourth derivative with respect to the $i^{th}$ dimension where $\xi(x)$ is a point determined by the error formula for Hermite Interpolation. 

\end{lemma}

\begin{proof}
  By Lemma \ref{lemma:sr_exact}, Simpson's Rule is exact for polynomials of degree up to 3. Now recall that we let $V_d=\{(v_1,\ldots,v_d)\mid v_i\in\{a_i,\frac{a_i+b_i}{2},b_i\}\}$ be the set of points at which this method will query the oracle function. Then for any estimation $\sum_{v\in V_d}w_d(v)f(v)$ given by the Simpson's Rule method, we have that $\int_{x\in R} p_3(x)=\sum_{v\in V_d}w_d(v)f(v)$ for any polynomial $p_3$ of degree at most 3, such that $p_3(v)=f(v)$ for all $v\in V_d$. 
  
  Then we can represent the error of the Simpson's Rule estimation in terms of $p_3$ using the following equation.
  
  $$\int_{x\in R} f(x)-\sum_{v\in V_d}w_d(v)f(v)=\int_{x\in R} f(x)-\int_{x\in R} p_3(x)=\int_{x\in R}\left(f(x)-p_3(x)\right)$$
  
  Next, let $i\in\{1,\ldots,d\}$. Then by holding the other $d-1$ dimensions constant, Hermite Interpolation gives the following error formula which holds for any $x_i\in[a_i,b_i]$
  
  $$f_i(x_i)-p_{3,i}(x_i)=\frac{f_i^{(4)}(\xi(x))c}{4!}$$
  
  where $c$ is defined using the following formula.
  
  $$c=\int_{a_i}^{b_i}\left((x-a_i)(x-\frac{a_i+b_i}{2})(x-b_i)\right)^2dx=\frac{(b_i-a_i)^7}{840}$$
  
  Therefore, $c\in\left[0,\max_i\frac{(b_i-a_i)^7}{840}\right]$\\
  
  However, since the other $d-1$ dimensions are being held constant, the above expression holds for any $x_j\in\mathbb{R}$ where $i\ne j$ as long as $x_i\in[a_i,b_i]$ since the error formula for Hermite Interpolation only depends on $x_i$. Note that, by changing the constant values these other dimensions are being held to, the above expression can change since $f_i$ and $p_{3,i}$ are defined in terms of the constant values $x_j$ for $j\ne i$.\\
  
  Additionally, at any point $x=(x_1,\ldots,x_i,\ldots,x_d)\in[a_1,b_1]\times\ldots\times[a_d,b_d]$, we have that $f(x)=f_i(x_i)$ by simply holding the $d-1$ other dimensions constant to their values at this point. Now we let $p_3$ be defined by the polynomials constructed using Hermite Interpolation. Then, $p_3(x)=p_{3,i}(x)$, so since Simpson's Rule is exact for any polynomial of degree up to 3, we get the following formula for the absolute value of the error at any $x\in[a_1,b_1]\times\ldots\times[a_d,b_d]$.
  
  $$\left|f(x)-p_3(x)\right|=\left|f_i(x_i)-p_{3,i}(x_i)\right|=\left|\frac{f_i^{(4)}(\xi(x))c}{4!}\right|$$\\
  
  Therefore, by taking the supremum of this formula over all dimensions $i$ and over all points $x^*\in[a_1,b_1]\times\ldots\times[a_d,b_d]$, we can upper bound the absolute value of the error. Note that, since $c\geq 0$, we can factor it out of the absolute value.
  
  $$\left|f(x)-p_3(x)\right|\leq\sup_{i,x^*}\left|\frac{f_i^{(4)}(\xi(x^*))c}{4!}\right|= \frac{c}{4!}\sup_{i,x^*}\left|f_i^{(4)}(\xi(x^*))\right|$$
  
  Now, since the absolute value of an integral of a function is less than or equal to the integral of the absolute value of the function, we can upper bound the error of the Simpson's Rule estimate to prove the theorem. Note that $x^*$ depends on the supremum, not the integral, so the supremum can be factored out of the integral. Additionally, the volume of the region of integration $[a_1,b_1]\times\ldots\times[a_d,b_d]$ is $\prod_{i=1}^d (b_i-a_i)$.
  
  \begin{align*}
    \left|\int_{x\in R} f(x)-\sum_{v\in V_d}w_d(v)f(v)\right| &= \left|\int_{x\in R}\left(f(x)-p_3(x)\right)\right|\\
    &\leq \int_{x\in R}\left|f(x)-p_3(x)\right|\\
    &\leq \int_{x\in R} \frac{c}{4!}\sup_{i,x^*}\left|f_i^{(4)}(\xi(x^*))\right|\\
    &= \frac{c}{4!}\sup_{i,x^*}\left|f_i^{(4)}(\xi(x^*))\right| \int_{x\in R} 1\\
    &= \frac{c}{4!}\sup_{i,x^*}\left|f_i^{(4)}(\xi(x^*))\right| \prod_{i=1}^d (b_i-a_i)
  \end{align*}
  
\end{proof}

\subsection{Finding the Sample Complexity Upper Bound}

We can now use the error bound from Lemma \ref{lemma:sr_error_term} to find a sample complexity upper bound on the error of the estimate produced by the Simpson's Rule method. 

\begin{theorem}\label{theorem:sr_upper_bound}
  If $\left|f_i^{(4)}(x)\right|\leq K$ and $b_i-a_i\leq B$ for all $i\in\{1,\ldots,d\}$ and $x\in[a_1,b_1]\times\ldots\times[a_n,b_n]$, then the error for the Simpson's Rule method has the following upper bound.
  
  $$\epsilon(SR,F,\phi)\leq \frac{3^{d/2}B^d\sigma}{2^{d/2-1}\sqrt{T}}+\frac{B^{d+7}}{840\cdot 4!}K$$
  
  Likewise, if $K=0$, then we get the following upper bound.
  
  $$\epsilon(SR,F,\phi)\leq \frac{3^{d/2}B^d\sigma}{2^{d/2-1}\sqrt{T}}$$

\end{theorem}

\begin{proof}
  Let $S=\int_{x\in R} f(x)-\sum_{v\in V_d}w_d(v)f(v)$ be the error for a non-noisy estimation produced by the Simpson's Rule method. 
  
  Then since $\mathbf{E}[\phi(x,f)]=f(x)$ because the oracle function is unbiased, we can use the linearity of expectation to get
  
  $$\mathbf{E}\left[\hat{f}(x)\right]=\mathbf{E}\left[\frac{1}{m}\sum_{i=1}^m\phi(x,f)\right]=\frac{1}{m}\sum_{i=1}^m\mathbf{E}[\phi(x,f)]=\mathbf{E}[\phi(x,f)]=f(x)$$
  
  Therefore, we can conclude that $\hat{f}(x)$, the average of the $m$ calls to the noisy oracle function, is also unbiased. Then we can use this fact to get the following expression.
  
  $$\mathbf{E}_\phi\left[\sum_{v\in V_d}w_d(v)\hat{f}(v)+S\right]=\mathbf{E}_\phi\left[\int_{x\in R} f(x)\right]+\mathbf{E}_\phi\left[\sum_{v\in V_d}w_d(v)\hat{f}(v)-\sum_{v\in V_d}w_d(v)f(v)\right]=\int_{x\in R} f(x)+0=\int_{x\in R} f(x)$$
  
  Additionally, since each call to the oracle function is independent, we can consider the variance of $\hat{f}(x)$.
  
  \begin{align*}
    \text{Var}\left(\hat{f}(x)\right) &= \text{Var}\left(\frac{1}{m}\sum_{i=1}^m\phi(x,f)\right)\\
    &= \frac{1}{m^2}\text{Var}\left(\sum_{i=1}^m\phi(x,f)\right)\\
    &= \frac{1}{m^2}\sum_{i=1}^m\text{Var}\left(\phi(x,f)\right)\\
    &\leq \frac{1}{m^2}\sum_{i=1}^m \sigma^2\\
    &= \frac{m\sigma^2}{m^2}\\
    &= \frac{\sigma^2}{m}
  \end{align*}
  
  Therefore, we can find the variance of $\sum_{v\in V_d}w_d(v)\hat{f}(v)+S$. 
  
  \begin{align*}
    \text{Var}\left(\sum_{v\in V_d}w_d(v)\hat{f}(v)+S\right) &= \text{Var}\left(\sum_{v\in V_d}w_d(v)\hat{f}(v)\right)\\
    &= \sum_{v\in V_d} \left(w_d(v)\right)^2 \text{Var}(\hat{f}(v))\\
    &\leq \frac{\sigma^2}{m}\sum_{v\in V_d} \left(w_d(v)\right)^2\\
    &= \frac{\sigma^2}{m}\prod_{i=1}^d\frac{(b_i-a_i)^2}{36}\sum_{v\in V_d} 16^{m_d(v)}\\
    &\leq \frac{\sigma^2}{m}\frac{B^{2d}}{36^d}\sum_{i=0}^d 16^i\binom{d}{i}2^{d-i}\\
    &= \frac{\sigma^2}{m}\frac{B^{2d}}{36^d}\sum_{i=0}^d \binom{d}{i}2^{d+3i}\\
    &= \frac{\sigma^2}{m}\frac{B^{2d}}{36^d}18^d\\
    &= \frac{\sigma^2}{m}\frac{B^{2d}}{2^d}\\
  \end{align*}
  
  Next, since $\mathbf{E}_\phi\left[\sum_{v\in V_d}w_d(v)\hat{f}(v)+S\right]=\int_{x\in R} f(x)$, Chebyshev's Inequality gives the following bound.
  
  $$\mathbf{P}\left(\left|\sum_{v\in V_d}w_d(v)\hat{f}(v)+S-\int_{x\in R} f(x)\right|>\epsilon\right)\leq \min\left(\frac{B^{2d}\sigma^2}{m2^d\epsilon^2},1\right)$$
  
  Now, let $X=\left|\sum_{v\in V_d}w_d(v)\hat{f}(v)+S-\int_{x\in R} f(x)\right|$. Then the Layer Cake representation gives the following expression for $\mathbf{E}[X]$.
  
  \begin{align*}
    \mathbf{E}\left[X\right] &= \int_0^\infty \mathbf{P}\left(X>\alpha\right)d\alpha\\
    &\leq \int_0^\infty \min\left(\frac{B^{2d}\sigma^2}{m2^d x^2},1\right)\\
    &\leq \int_0^{B^d\sigma/\sqrt{m2^d}}1d\alpha+\frac{B^{2d}\sigma^2}{m2^d}\int_{B^d\sigma/\sqrt{m2^d}}^\infty \frac{1}{\alpha^2} d\alpha\\
    &= \frac{B^d\sigma}{\sqrt{m2^d}}+\frac{B^{2d}\sigma^2}{m2^d}\frac{\sqrt{m2^d}}{B^d\sigma}\\
    &= 2\frac{B^d\sigma}{\sqrt{m2^d}}\\
    &= \frac{B^d\sigma}{2^{d/2-1}\sqrt{m}}
  \end{align*}
  
  Next, since $\mathbf{E}_\phi\left[\left|\sum_{v\in V_d}w_d(v)\hat{f}(v)+S-\int_{x\in R} f(x)\right|\right]=\mathbf{E}[X]\leq \frac{B^d\sigma}{2^{d/2-1}\sqrt{m}}$, we can get the following upper bound.
  
  $$\mathbf{E}_\phi\left[\left|\sum_{v\in V_d}w_d(v)\hat{f}(v)-\int_{x\in R} f(x)\right|\right]\leq \frac{B^d\sigma}{2^{d/2-1}\sqrt{m}}-S$$
  
  However, Lemma \ref{lemma:sr_error_term} tells us that $\left|S\right|\leq \frac{c}{4!}\sup_{i,x^*}\left|f_i^{(4)}(\xi(x^*))\right|\prod_{i=1}^d(b_i-a_i)$. Therefore, we can substitute this formula into the above inequality to get a new upper bound.
  
  $$\mathbf{E}_\phi\left[\left|\sum_{v\in V_d}w_d(v)\hat{f}(v)-\int_{x\in R} f(x)\right|\right]\leq \frac{B^d\sigma}{2^{d/2-1}\sqrt{m}}+\frac{c}{4!}\sup_{i,x^*}\left|f_i^{(4)}(\xi(x^*))\right|\prod_{i=1}^d(b_i-a_i)$$
  
  Then since $\left|f_i^{(4)}(x)\right|\leq K$ for all $x\in[a_1,b_1]\times\ldots\times[a_d,b_d]$, $b_i-a_i\leq B$ for all $i\in\{1,\ldots,d\}$, and $c\in\left[0,\max_i\frac{(b_i-a_i)^7}{840}\right]\subseteq[0,\frac{B^7}{840}]$, we get the following inequality.
  
  $$\mathbf{E}_\phi\left[\left|\sum_{v\in V_d}w_d(v)\hat{f}(v)-\int_{x\in R} f(x)\right|\right]\leq \frac{B^d\sigma}{2^{d/2-1}\sqrt{m}}+\frac{B^{d+7}}{840\cdot 4!}K$$
  
  Finally, since $T=m3^d$, we can substitute $m=T/3^d$ into the above formula to get the following inequality, thus proving the first part of Theorem \ref{theorem:sr_upper_bound}.
  
  $$\epsilon(SR,F,\phi)\leq \frac{3^{d/2}B^d\sigma}{2^{d/2-1}\sqrt{T}}+\frac{B^{d+7}}{840\cdot 4!}K$$
  
  Then it is trivial to see that, if $K=0$, the other bound in the theorem holds. Therefore, we have proven Theorem \ref{theorem:sr_upper_bound}.
  
  $$\epsilon(SR,F,\phi)\leq \frac{3^{d/2}B^d\sigma}{2^{d/2-1}\sqrt{T}}$$
  
\end{proof}
\section{Comparison of Sample Complexity Upper Bound Between Gaussian Quadrature and Simpson's Rule}
\label{sec:gqsrcomparison} 

Now that we have acquired a sample complexity upper bound for both Gaussian Quadrature and Simpson's Rule, we can compare the bounds from Theorems \ref{theorem:gq_upper_bound} and \ref{theorem:sr_upper_bound} by setting $B=2r$ in Simpson's Rule since Gaussian Quadrature requires the integration be performed over the region $[-r,r]^d$. Recall the following upper bound for Gaussian Quadrature from Theorem \ref{theorem:gq_upper_bound}.

$$\epsilon(GQ,F,\phi)\leq \frac{2^{d+1}r^d\sigma}{\sqrt{T}}+\frac{2^{d+1} r^{d+5}}{6\cdot 45}K$$

Then, recall the upper bound for Simpson's Rule from Theorem \ref{theorem:sr_upper_bound}. By substituting $B=2r$ into this formula, we get the following upper bound for Simpson's Rule which can then be compared to Gaussian Quadrature.

\begin{align*}
    \epsilon(SR,F,\phi) &\leq \frac{3^{d/2}B^d\sigma}{2^{d/2-1}\sqrt{T}}+\frac{B^{d+7}}{840\cdot 4!}K\\
    &= \frac{3^{d/2}2^d r^d\sigma}{2^{d/2-1}\sqrt{T}}+\frac{2^{d+7}r^{d+7}}{840\cdot 4!}K\\
    &= \frac{3^{d/2}2^{d/2+1} r^d\sigma}{\sqrt{T}}+\frac{2^{d+1}r^{d+7}}{7\cdot 45}K
\end{align*}

Therefore, we see that Simpson's Rule needs to make $3^d$ queries for every $2^d$ queries made by Gaussian Quadrature in order for the noise from the oracle function to affect both methods the same as, if $T_{GQ}=m2^d$ while $T_{SR}=m3^d$, the two bounds have matching first terms. Additionally, the error due to the methods' estimations only differs by a factor of $6r^2/7$. As such, we can conclude that these methods behave similarly over equivalent regions. However, Simpson's Rule needs to query more points to achieve the same error rate as Gaussian Quadrature. As such, the main benefit of Simpson's Rule is simply that it can support arbitrary rectangular regions while Gaussian Quadrature is able to use fewer points to estimate the integral over a strictly cubic region.

\end{document}